\documentclass[draft, giq]{geom}

\newcommand{\diff}{\mathrm{d}}
\begin{document}
\head{35}{2026} \thispagestyle{plain}

\title[Deformation of Affine Structures and the Cohomology of Koszul-Vinberg \ldots] {Deformation of affine
structures and the cohomology of Koszul-Vinberg algebras on the Lie
groups SO(2),   $\mathrm{H_{3}}(\mathbb{R})$ and Galilei group
SGal(3)}

\author[P.  Assandje, N. Pefoukeu, D. Ngaha,  F.
 Barbaresco and  N.  Boyom] {Prosper  M. Assandje,  Romain  N. Pefoukeu, Michel B. Ngaha,  Frederic
Barbaresco \and  Michel N.  Boyom}
\date{}
\maketitle

\smallskip
\begin{abstract}
In this work, we compare the De Rham Cohomology and the Koszul
Vinberg (KV) Cohomology groups on the Lie groups SO(2),
$\mathrm{H_{3}}(\mathbb{R})$ and SGal(3).  We model their
interactions by constructing a three vertex directed graph
connecting associative algebras, KV-Cohomology, and Lie groups. By
computing the exact dimensions of these complexes, we evaluate their
algebraic quotient, which measures the gap separating global
topological invariants from left-invariant flat affine structures.
Extending this geometric framework to the coadjoint orbits of the
Heisenberg group $\mathrm{H}_{3}(\mathbb{R})$ and the Galilei group
$\mathrm{SGal}(3)$, we investigate their properties under an
invariant Lagrangian foliation inherited from a constant rank
Nijenhuis endomorphism preserving the Boyom complex. Finally, we
establish a vanishing theorem for the second KV-Cohomology group. We
demonstrate that any infinitesimal deformation of the affine
structure governed by the polarized Maurer Cartan equation is
trivial, thereby proving the structural rigidity of these orbits.\\[0.2cm]
\textsl{MSC}: 14B10, 14B12,  14B15,  13D03,  13D10,   14D15, 13D45,
53A15
\\
\textsl{Keywords}: coadjoint orbit,  Cobord,  Cochain complex,
Koszul Vinberg algebra, Koszul Vinberg Cohomology.
\end{abstract}

    \tableofcontents{}

\label{first}
\section{Introduction\label{sec:intro}}

Koszul-Vinberg Cohomology or  KV-Cohomology is an affine algebraic
construction that associates a cochain complex to a convex affine
structure. Its Cohomology makes it possible to detect and classify
deformations of affine convex structures, and obstructions to the
flatness of a connection. Souriau
\cite{souriau,barbaresco3,souriau1,barbaresco,barbaresco1,barbaresco2,barbaresco2025,barbaresco2025a,ma},
working within the framework of symplectic geometry and Lie groups,
introduced a non-trivial cocycle, known as the Souriau cocycle,
which is related to the Kirillov-Kostant-Souriau (KKS) form on
coadjoint orbits. This cocycle makes it possible to define central
extensions of Lie groups through group Cohomology. In information
geometry, this cocycle acts as a bridge between the algebraic
structure (Lie group) and the geometric structure (differential
forms, metrics). The primary objective is to understand whether and
how certain structures can implicitly emerge on coadjoint orbits for
instance, through Lagrangian foliations whose leaves naturally carry
affine structures since a coadjoint orbit equipped with a Lagrangian
foliation yields an affine submanifold, through invariant flat
connections for Lie groups, and finally, through the study of
deformations controlled by KV-Cohomology that could induce affine
structures on substructures associated with the orbits. Given that
de Rham cohomology \cite{grot, chev} and KV-Cohomology are
isomorphic on the abelian group SO(2) but diverge on the non-abelian
Heisenberg group $\mathrm{H_{3}}(\mathbb{R})$, However, a deeper
obstruction appears when moving to the non-compact, non-abelian
Galilei group $\mathrm{SGal}(3)$, a fundamental geometric question
arises: how can we systematically unify the interplay between the
algebraic framework of associative algebras, the deformation theory
of KV-Cohomology, and the differential geometry of Lie groups and
how can this cohomological defect be resolved to guarantee the
stability of the underlying structures? This allows us to construct
a three vertex directed graph connecting associative algebras,
KV-Cohomology, and Lie groups.  Work in this field traces its
origins back even further. Elie
Cartan\cite{ct,barbaresco2026,barbaresco2023} does not explicitly
mention $\Lambda(\mathfrak{g}^*)$ (the complex of alternating forms
on a Lie algebra), because he treats groups as symmetric spaces and
is therefore interested in differential forms which are invariant
under both left and right translations, which corresponds to the
elements of $\Lambda(\mathfrak{g}^*)$ invariant by the prolongation
of the coadjoint representation. Nevertheless, it can be said that
by 1929 an essential piece of the cohomological theory of Lie
algebras was in place. According to M. Gerstenhaber \cite{gers},
every restricted deformation theory generates its own corresponding
Cohomology theory. The deformation theory of associative algebras
and their modules involves Hochschild Cohomology, while the
deformation theory of Lie algebras relies on Chevalley Eilenberg
Cohomology\cite{chev}. This area of research has experienced
significant development. In particular, the Cohomology theory of
Koszul Vinberg algebras (KV Cohomology) was initiated by Albert
Nijenhuis \cite{kz3}, to study the deformations of locally flat
manifolds \cite{kz1}. This pioneering work was later extended
\cite{22} and rediscovered through modern conceptual frameworks
\cite{live}. In \cite{22}, Michel Nguiffo Boyom explored the
relationships between the Cohomology theory of Koszul Vinberg
algebras and various related geometric structures. He provided a
rigorous definition of the KV-complex and demonstrated how KV
Cohomology on a locally flat manifold $(M,\nabla)$ connects
Nijenhuis's original ideas to the Cohomology of higher-order
differential forms valued in $T^{*}M$. Furthermore, Boyom \cite{22},
established a relationship between the real KV-Cohomology of
Lagrangian foliations and the structure of Poisson manifolds an
their Dirac reductions an observation also noted by J. Stasheff in
private communication. In addition to its theoretical interest,
KV-Cohomology has  proven useful in the classification of short
modules. It plays a key role in the study of deformations and
reductions of Poisson manifolds, providing a bridge algebraic and
geometric frameworks. Indeed, the set of Casimir functions
corresponds to the zeroth-degree de Rham Cohomology. Consequently,
in Souriau's Lie group thermodynamics model, entropy which is a
Casimir function whose level sets are the coadjoint orbits is
directly linked to the de Rham Cohomology of zero degree. Vorob'ev
and Karasev \cite{vo} suggested a classification of Cohomology in
terms of closed forms and the de Rham Cohomology of coadjoint orbits
(Euler orbits), which are the symplectic leaves of a Poisson
manifold. Ping Xu \cite{pi}, observed that Poisson Cohomology
reflects two types of information about a Poisson manifold: the de
Rham Cohomology of the symplectic leaves and the variation of the
symplectic structures along these leaves. He demonstrated that in
cases where the symplectic foliations are trivial fibrations,
computing the Poisson Cohomology is equivalent to computing the de
Rham Cohomology of certain torus bundles.
Dazord\cite{da,da1,so,kar,dua1,amam}, developed an affine model that
integrates Souriau's affine structures. Specifically, he
demonstrated that affine Poisson groups coincide with the affine
structures introduced by  Souriau. Following this work, we
demonstrate that the Cohomology groups of the KV-Cohomology $
\mathrm{H^{0}_{KV}\left(\mathcal{A},\mathfrak{so}(2)\right)}\simeq\mathbb{R},\quad
\mathrm{H^{1}_{KV}\left(\mathcal{A},\mathfrak{so}(2)\right)}\cong\mathbb{R}$,\quad
$\mathrm{H^{2}_{KV}\left(\mathcal{A},\mathfrak{so}(2)\right)}=0$ and
The de Rham Cohomology groups $\mathrm{
H^0(\mathfrak{so}(2),\mathbb{R})}\simeq \mathbb{R},\quad
\mathrm{H^1(\mathfrak{so}(2),\mathbb{R})}=\mathbb{R}$, are
isomorphic on $\mathrm{SO(2)}$. We show that the Cohomology of the
Lie algebra \(\mathfrak{h}_{3}\) of left invariant forms, which
reflects the algebraic structure of the group, is given by $
\mathrm{H^{0}(\mathfrak{h}_{3},\mathbb{R})}\cong \mathbb{R},\quad
\mathrm{H^{1}(\mathfrak{h}_{3},\mathbb{R})}\cong
\mathbb{R}^{2},\quad
\mathrm{H^{2}(\mathfrak{h}_{3},\mathbb{R})}\cong
\mathbb{R}^{2},\quad
\mathrm{H^{3}(\mathfrak{h}_{3},\mathbb{R})}\cong \mathbb{R}$. The
Koszul-Vinberg Cohomology groups with coefficients in the adjoint
module are given by $
\mathrm{H_{KV}^{0}(\mathcal{A},\mathfrak{h}_{3})}\cong
\mathbb{R},\quad
\mathrm{H_{KV}^{1}(\mathcal{A},\mathfrak{h}_{3})}=\{0\},\quad
\mathrm{H_{KV}^{2}(\mathcal{A},\mathfrak{h}_{3})}\cong\mathbb{R}^{21}.$
We show that the Cohomology of the Lie algebra $\mathfrak{so}(3)$ of
left invariant forms, which reflects the algebraic structure of the
group, is given by
$\mathrm{H^{0}(\mathfrak{so}(3),\mathbb{R})}=\mathbb{R},\quad\mathrm{H^{1}(\mathfrak{so}(3),\mathbb{R})}=0,\quad
\mathrm{H^{2}(\mathfrak{so}(3),\mathbb{R})}=0,\quad\mathrm{H^{3}(\mathfrak{so}(3),\mathbb{R})}=\mathbb{R}.$
Let G = SGal(3) denote the Galilei group, we show that the de Rham
Cohomology groups of $G$ are isomorphic to those of the Lie algebra
$\mathfrak{so}(3)$ and we have $\mathrm{H^{0}_{dR}(G)} \simeq
\mathbb{R}, \quad \mathrm{H^{1}_{dR}(G)} = 0, \quad
\mathrm{H^{2}_{dR}(G)} = 0, \quad \mathrm{H^{3}_{dR}(G)} \simeq
\mathbb{R}$. The quotient of the KV-Cohomology by the de Rham
Cohomology is given by $\mathrm{Q^0} = 0, \quad \mathrm{Q^1} = 0$.
We show that the quotient of the Cohomology of the Lie algebra
\(\mathfrak{h}_{3}\) of left invariant forms, which reflects the
algebraic structure of the group on $\mathrm{H_{3}}(\mathbb{R})$ is
given by $\mathrm{Q^0} = 0, \quad \mathrm{Q^1} =
\mathbb{R}^{-2},\quad \mathrm{Q^2} = \mathbb{R}^{19}$. Given
$\mathcal{O} = \{x+xy, y-xz, z\}$ be a coadjoint orbit equipped with
a Koszul-Vinberg structure inherited from an invariant Lagrangian
polarization $F = \left\{x+xz, y , z\right\}$ with
$y,z=\text{const}$ with $f,g,h\in C^{\infty}(\mathcal{O})$, and
given $C(\mathrm{H_{3}}(\mathbb{R}))$ be the commutative associative
algebra of differentiable functions on $\mathrm{H_{3}}(\mathbb{R})$.
We prove that, if $P$ is a constant-rank Nijenhuis endomorphism
acting on the bundle and preserving the filtration of the Boyom
complex, then the second Cohomology group of the KV-Cohomology
vanishes relative to the Maurer-Cartan polarization. Furthermore,
any infinitesimal deformation of the affine structure satisfying the
polarized Maurer-Cartan equation $[P, P](f, g) = 0$ is equivalent to
the initial structure and  that the formal product of two functions
is given by $f\star g=\exp\left(h \{f,g\}\right)$, where $\{.,.\}$
denotes the Poisson bracket. We prove that the de Rham Cohomology
groups of $G$ are isomorphic to those of the rotation group SO(3) $
\mathrm{H}^{0}_{dR}(G) \simeq \mathbb{R}, \quad
\mathrm{H}^{1}_{dR}(G) = 0, \quad \mathrm{H}^{2}_{dR}(G) = 0, \quad
\mathrm{H}^{3}_{dR}(G) \simeq \mathbb{R}$. Given $g = (A, b, c,
e)\in \mathrm{SGal(3)}$, $\xi=(\Xi,\vartheta,\nu,\varepsilon)\in
\mathfrak{sgal}(3)$, and $\mu = (j, k, p, E, m) \in
\mathfrak{sgal}(3)^*$. Given $\Omega=\left\{\mu = (j, k, p, E, m)
\in \mathfrak{sgal}(3)^*;\quad m>0\right\}$ the Koszul cone on
Galilei group, and $\Omega^*=\left\{\mu = (j, k, p, E, m) \in
\mathfrak{sgal}(3)^*;\quad m>0,\textrm{and}\quad
E-\frac{\|p\|^2}{2m}\right\}$ the Koszul dual cone on Galilei group.
We show that on Souriau coadjoint orbite on Galilei group given by
$\mathcal{O}=\left\{m, U,S^2\right\}$ where $m$ is a parameter of
the central extension of the Galilei group,
$U=E-\frac{\|p\|^2}{2m}$, $S=j-\frac{1}{m}(j\times p)$ constant,
there exist a lagrangian foliation linked to the
Kirillov-Konstant-Souriau with the symplectic structure
$\omega=\sum_{i=1}^{3}\diff q_{i}\wedge \diff p_{i}$ with
$p=\left\{P_{1},P_{2},P_{3}\right\}$ and
$q=\left\{K_{1},K_{2},K_{3}\right\}$ where the Lagrangian leaves is
given by $\mathcal{F}=\left\{(q,p);\quad  q\in \mathbb{R}^{3},\quad
and\quad p \quad fixed\right\}$. After the introduction, the first
section recalls the preliminaries. In section $3$, we determine the
KV-Cohomology and de Rham Cohomology on SO(2). In section $4$, we
determine the KV-Cohomology and de Rham Cohomology on
$\mathrm{H_{3}}(\mathbb{R})$. In section $5$, we construct the
quotient of the scalar KV-complex by the de Rham complex and propose
a corrective quotient. In section $6$, we presents Fedosov's
deformation quantization. Finally, section $7$ presents the
conclusion.

\section{Preliminaries\label{sec:start}}
In this section we recall the preliminaries notion on Koszul-Vinberg
Cohomology, de Rham Cohomology, and deformation.
\subsection{KV-Algebra\label{ssec:math}}

Let $\mathbb{F}$ be a commutative field of characteristic zero,
$\mathcal{A}$ be an algebra over  $\mathbb{F}$. The product of two
elements $a, b \in \mathcal{A}$ is denoted by $ab$, and
\begin{equation*}
(a, b, c) = (ab)c - a(bc)\end{equation*} is the associator  of $a,
b,c \in \mathcal{A}$.

By setting $KV(a, b, c) =(a, b, c)-(b,a, c)$ we have the following
definition
\begin{definition}\textup{\cite{69}}\label{kva}
An algebra $\mathcal{A}$ is called a Koszul-Vinberg algebra, or
$\mathrm{KV}$-algebra, if $(a, b, c) = (b, a, c),\; for\; all\; a,
b, c \in \mathcal{A}$ i.e., $KV(a, b, c) =0$. A KV-algebra is  also
known as a left-symmetry algebra or pre-Lie algebra.
\end{definition}

\begin{proposition}\textup{\cite{69}}
Let $S$ a Riemannian manifold. If $(S,\nabla)$ be a locally flat
manifold, and if  define on $\mathfrak{X}(S)$ the product: $X\star
Y=\nabla_{X}Y,\; \forall\; X,Y\in \mathfrak{X}(S)$ then the pair
$(\mathfrak{X}(S),\nabla)$ is a $\mathrm{KV}$-algebra and will be
called the $\mathrm{KV}$-algebra of the locally flat manifold
$(S,\nabla)$.
\end{proposition}

\begin{definition}\textup{\cite{69}}
The subspace $J (\mathcal{A})$ of Jacobi elements of a
$\mathrm{KV}$-algebra $\mathcal{A}$ is the subset of
$\xi\in\mathcal{A}$ satisfying the identity $(a, b, \xi ) = 0$ for
all $a, b\in \mathcal{A}$. Actually $J (\mathcal{A})$ is an
associative subalgebra containing the center of $\mathcal{A}$.
\end{definition}

\subsection{KV-Module\label{ssec:math1}}
In \cite{69}, let $\mathcal{A}$ be a $\mathrm{KV}$-algebra. We
consider a vector space $W$ or $\ker\mathbb{F}$ with two bilinear
maps $\mathcal{A}\times W\longrightarrow W,\quad  (a, w) \longmapsto
aw$ and $W\times\mathcal{A}\longrightarrow W,\quad  ( w,a)
\longmapsto wa$. Given $a, b\in \mathcal{A}$ and $w \in W$ one sets
$(a, b, w) = (ab)w - a(bw),\quad (a, w, b) = (aw)b - a(wb)$, and
$(w, a, b) = (wa)b - w(ab)$.

\begin{definition}\textup{\cite{69,22}}
A vector space $W$ with bilinear maps as above is called a (real or
complex) two-sided $\mathcal{A}$-$\mathrm{KV}$-module if
\begin{equation*}
(a, b, w) = (b, a, w)\quad and\quad (a, w, b) = (w, a,
b).\end{equation*} A left (right) $\mathrm{KV}$-module over
$\mathcal{A}$ is a $\mathrm{KV}$-module $W$ whose right (left)
$\mathcal{A}$-action is trivial, meaning that $wa = 0 \quad (aw =
0)$, for all $(w, a) \in W \times \mathcal{A}$.
\end{definition}

\begin{definition}\textup{\cite{69,22}}
The subspace J (W ) of Jacobi elements of a $\mathrm{KV}$-module $W$
consists of $w\in W$ satisfying $(a, b, w) = 0$ for all $a, b \in
\mathcal{A}$.\end{definition}

\begin{proposition}\textup{\cite{69,22}}
Let $(S, \nabla)$ is a flat locally manifold. The $C^{\infty}(S)$
space is a left $\mathrm{KV}$ -module on $(\mathcal{A},\nabla)$ with
the action
\begin{equation*}
X.f:= X(f ) = \diff f (X ) \qquad \textrm{for  all}\quad X \in
\mathfrak{X}(S),\qquad  \textrm{for  all} \quad f \in C^{\infty}(S).
\end{equation*}
\end{proposition}

\subsection{Koszul Vinberg Cohomology or Boyom Cohomology\label{ssec:math2}}
Let $\mathcal{A}$ be a KV-algebra $W$ a KV-module.
\begin{definition}\textup{\cite{69}}
 The $q^{th}$ space $C^{q}(\mathcal{A},W)$ on the $\mathrm{KV}$-algebra
$\mathcal{A}$ is given by
 \begin{eqnarray*}
 C^{q}(\mathcal{A},W)&=&\left\{
                  \begin{array}{ll}
                    0 & \hbox{if q $<$ 0} \\
                    J(W) & \hbox{if q=0} \\
                    \mathrm{Hom}_{\mathbb{R}}\left(\otimes^{q}\mathcal{A},W\right), & \hbox{if q $\geq$ 1.}
                  \end{array}
                \right.
 \end{eqnarray*}
The $\mathrm{KV}$-complex is defined by
$C(\mathcal{A},W)=\oplus_{q\in \mathbb{Z}}C^{q}(\mathcal{A},W)$.
$C(\mathcal{A},W)$ is called the $\mathcal{A}$-$\mathrm{KV}$
graduated modulus.\end{definition}

\begin{definition}\textup{\cite{69}}\label{k1} We define the Cobord operator
\newline $\delta_{KV}^{q}:C^{q}(\mathcal{A},W)\longrightarrow
C^{q+1}(\mathcal{A},W)$. The coboundary $\delta_{KV}^{q}f\in
 C^{q+1}(\mathcal{A},W)$ and is given
 by \begin{eqnarray*} \left\{
                   \begin{array}{ll}
                     \delta_{KV}^{0} f(a)= -af +fa,\quad a \in \mathcal{A}, \quad\forall f\in C^{0}(\mathcal{A},W)& \hbox{} \\
                      & \hbox{} \\
                     \delta_{KV}^{q}f\left(a_{1},\dots,
 a_{q+1}\right)=\sum_{j=1}^{q}(-1)^{j}\left(\left(a_{j}f\right)\left(a_{1},\dots,
\hat{a}_{j},\dots, a_{q+1}\right)\right.& \hbox{} \\
\left.+\left(f\left(a_{1}\dots,\dots, \hat{a}_{j},\dots
,a_{q},a_{j}\right)\right).a_{q+1}\right) &\hbox{}
                   \end{array}
                 \right.
 \end{eqnarray*}for\;
$\left(a_{1},\dots,
 a_{q+1}\right)\in \mathcal{A}^{q+1}$.
\end{definition}

\begin{definition}\textup{\cite{69}}
 Set $C^{0}(\mathcal{A},W) = J (W)$ and $\delta_{KV}^{0}f(a) = -af +fa$ for all $a \in \mathcal{A}$ and all $f \in J(W)$.
\end{definition}

\begin{theorem}\textup{\cite{69}} The pair $(C(\mathcal{A},W), \delta_{KV})$ cochains complex where,
 $C(\mathcal{A},W)=\bigoplus_{q\in \mathbb{Z}}C^{q}(\mathcal{A},W)$
 is the $\mathcal{A}$-$\mathrm{KV}$ graduated modulus and $\delta_{KV}$ the
 cobord operator on $C(\mathcal{A},W)$ is a cochain complex.
\end{theorem}
\begin{definition}\textup{\cite{69}} The complex $(C(\mathcal{A},W), \delta_{KV})$ is called the
complex of the \newline Koszul-Vinberg cohomology of the KV-algebra
$\mathcal{A}$ with values in $W$.
\end{definition}
Since $\delta_{KV}^{q}\circ \delta_{KV}^{q-1}=0$  for all
$q\in\mathbb{Z},$ then \;$\mathrm{Im\delta_{KV}^{q-1}} \subset\ker
\delta_{KV}^{q}$ and we have the following definition.

\begin{definition}\textup{\cite{69}}
 The Cohomology of the cochain complex $(C(\mathcal{A},W), \delta_{KV})$ is
 called the Koszul-Vinberg Cohomology (or simply $\mathrm{KV}$-Cohomology) of the
  $\mathrm{KV}$-algebra
$(\mathcal{A},W)$ with values in $W$ and is denoted
\begin{eqnarray*}\mathrm{H_{KV}(\mathcal{A},W)}&=&\oplus_{q\in \mathbb{Z}}\mathrm{H_{KV}^{q}(\mathcal{A},W)}\end{eqnarray*} where
\begin{eqnarray*}\mathrm{H_{KV}^{q}(\mathcal{A},W)}=\frac{\ker \delta_{KV}^{q}}{\mathrm{Im \delta_{KV}^{q-1}}}\cdot\end{eqnarray*}

\end{definition}

\subsection{De Rham Complex and De Rham Cohomology Groups\label{ssec:math3}}

\begin{definition}\textup{\cite{rham}}\label{rham}
Let $M$ be a manifold of class $C^{\infty}$. Let $\mathcal{A}(M)=
\bigoplus_{i} \Omega^{i}(M)$ be the $\mathbb{Z}$-graded algebra of
differential forms on $M$ (with complex coefficients). Let us start
from the de Rham Cohomology defined from the exterior
differentiation $\diff: \Omega^{\bullet}(M)\longrightarrow
\Omega^{\bullet+1} (M)$. The differential $\diff$ is an operator
satisfying the following conditions
\begin{enumerate}[\textup{\arabic*.}]
\item Antiderivation: $\diff(\omega \wedge \eta) = \diff\omega \wedge \eta + (-1)^{deg \omega} \omega \wedge
 \diff\eta$ for $\omega \in
\Omega^{k}(M)$
\item Nilpotence: $\diff^2 = 0$
\item On functions: for $f \in \Omega^{0}(M)=C^{\infty}(M)$, $\diff f$ is the usual differential
\item $(\diff f)(\xi)=\xi f$ where $f\in C^{\infty}(M)$, $\xi$ being a vector field on
$M$. We have the equality \begin{eqnarray*} \diff
\omega\left(\xi_{0},\dots,
 \xi_{k}\right)
 &=&\sum_{i=0}^{k}(-1)^{i}\xi_{i}\left(\omega\left(\xi_{0},\dots,
\hat{\xi}_{i},\dots,\xi_{k}\right)\right)\\
&&+ \sum_{0\leq i<j\leq
k}(-1)^{i+j}\omega\left([\xi_{i},\xi_{j}],\xi_{0},\dots,
\hat{\xi}_{i},\dots,\hat{\xi}_{j},\dots,
\xi_{k}\right)\end{eqnarray*}
\end{enumerate}
The relation $\diff^2 = 0$ allows us to define a Cohomology
$\mathrm{H^{*} (M,\mathbb{C})}:= \frac{Ker\diff }{Im\diff}$ , that
is, for each $i$, $\mathrm{H^{i} (M,\mathbb{C})}:=
\frac{Ker\left(\diff\mid_{\omega^{i}(M)}\right)
}{Im\left(\diff\mid_{\omega^{i-1}(M)}\right)}$ called of the De Rham
Cohomology.
\end{definition}

\subsection{Invariant Affine Structures on a Lie Group and Poisson Bracket\label{ssec:math4}}

\begin{remark}
Let $G$ be a Lie group, we have the following formulas
\begin{enumerate}
    \item For one-form $\omega$ on G, \begin{equation}\label{f1}
\diff\omega(X, Y) = X(\omega(Y)) - Y(\omega(X)) - \omega([X, Y])
\end{equation} where $\diff$ is the de Rham differential operator, where  $X,
Y$  are  vector  fieds on $G$
    \item The Cartan-Schouten connection is defined   by
\begin{equation}\label{f2}
\nabla_X Y = \frac{1}{2}[X, Y]\end{equation} for all left-invariant
vector fields $X$ and $Y$ on $G$.
\item For a left-invariant one-form $\omega$ and a vector field $X$, the
action of the connection is given by
\begin{equation}\label{f3}
(\nabla_X \omega)(Y) = X(\omega(Y)) - \omega(\nabla_X Y)
\end{equation}.
\end{enumerate}

\end{remark}

\begin{definition}\textup{\cite{Dongho}}
A Poisson bracket on an algebra $\mathcal{A}$ is a bilinear mapping
$\{\cdot, \cdot\} : \mathcal{A} \times \mathcal{A} \to \mathcal{A}$
satisfying the following conditions
\begin{enumerate}[\textup{\arabic*.}]
    \item $\{a, b\} = -\{b, a\}$ \quad (\textrm{Antisymmetry})
    \item $\{a, \{b, c\}\} + \{b, \{c, a\}\} + \{c, \{a, b\}\} = 0$ \quad  (\textrm{Jacobi
    identity})
    \item $\{a, bc\} = b\{a, c\} + c\{a, b\}$ \quad  (\textrm{Leibniz
    rule})
\end{enumerate}
The pair $(\mathcal{A}, \{\cdot, \cdot\})$ is called a Poisson
algebra.
\end{definition}
In the following, all manifolds are assumed to be connected and of
class $C^\infty$. If $M$ is a manifold, $C^\infty(M)$ denotes the
space of $C^\infty$ functions from $M$ to $\mathbb{R}$.
\begin{proposition}\label{sc}\textup{\cite{sc}}
Let $V$ be a Poisson manifold i.e., a manifold equipped with a
Poisson structure. Then, there exists a unique antisymmetric
contravariant two-tensor $\barwedge$ on $M$ such that
\begin{equation}\label{poi}
\{f, g\}(x) = \barwedge(x)(\diff f_x, d g_x), \quad \textrm{for
all}\quad f, g \in C^\infty(M).
\end{equation}
The tensor defined in this way is called the Poisson tensor (or the
structure tensor) of the Poisson manifold. Conversely, if we define
the bracket $\{\cdot, \cdot\}$ on $C^\infty(M)$ using equation
(\ref{poi}), the Jacobi identity for this bracket is equivalent to
the vanishing of the Schouten-Nijenhuis bracket
\begin{equation}
[\barwedge, \barwedge] = 0.
\end{equation}
\end{proposition}

\begin{definition}\textup{\cite{zeul}}
The canonical Lie  Poisson bracket (also known as the Berezin
Kirillov  Souriau bracket) on the dual space $\mathfrak{sgal}(3)^*$
for any two smooth functions $f, h \in
\mathcal{C}^{\infty}(\mathfrak{sgal}(3)^*)$ at a point $\mu \in
\mathfrak{sgal}(3)^*$ is defined by
\begin{equation}
\{f, h\}(\mu) = \langle \mu, [\,\diff f_\mu, \diff h_\mu\,] \rangle
\end{equation}
where $\diff f_\mu, \diff h_\mu \in \mathfrak{sgal}(3)$ represent
the differentials of the functions identified as elements of the Lie
algebra, and $[\cdot, \cdot]$ denotes the standard Lie bracket.
\end{definition}

\begin{definition}\textup{\cite{zeul}}\label{poi}
For the coordinate functions on the dual space
$\mathfrak{sgal}(3)^*$, which correspond directly to the elements of
the dual basis $\mu(X) = x$ for any generator $X \in
\mathfrak{sgal}(3)$, the general Lie-Poisson bracket simplifies to
the evaluation of the Lie bracket of their respective matrix
generators
\begin{equation}
\{x, y\}(\mu) = \mu([X, Y])
\end{equation}
where $x, y \in \mathfrak{sgal}(3)^*$ are the linear coordinate
functions associated with the Lie algebra generators $X, Y \in
\mathfrak{sgal}(3)$.
\end{definition}

\subsection{Matrix Representation of the Galilei Group SGal(3)}

\begin{definition}\label{gal}\textup{\cite{zeul}}  The special Galilei group $\mathrm{SGal(3)}$ is a 10-dimensional non-compact
connected Lie group. An element $g \in \mathrm{SGal(3)}$ is
characterized by a rotation $A \in \mathrm{SO(3)}$, a boost $b \in
\mathbb{R}^3$, a spatial translation $c \in \mathbb{R}^3$, and a
time translation $e \in \mathbb{R}$. The standard $5 \times 5$
matrix representation is given by
\begin{equation*}
g=(A, b, c, e) =
\begin{pmatrix}
A & b & c\\
0_{1 \times 3} & 1 & e \\
0_{1 \times 3} & 0 & 1
\end{pmatrix} \in GL(5, \mathbb{R})
\end{equation*}
The group law is defined by matrix multiplication, and the inverse
element is
\begin{equation*}
g^{-1} =
\begin{pmatrix}
A^{-1} & -A^{-1}b& -A^{-1}(c - be) \\
0_{1 \times 3} & 1 & -e \\
0_{1 \times 3} & 0 & 1
\end{pmatrix}.
\end{equation*}
\end{definition}

\begin{proposition}\textup{\cite{zeul}}
The dual space $\mathfrak{sgal}(3)^*$ consists of $5 \times 5$
matrices representing the physical momenta of the system. An element
$\mu \in \mathfrak{sgal}(3)^*$ is given by
\begin{equation*}
\mu = \begin{pmatrix}
j & k & p \\
0_{1\times3} & 0 & E \\
0_{1\times3} & 0 & 0
\end{pmatrix} \in \mathfrak{gl}(5, \mathbb{R})
\end{equation*}
where the components correspond to: $j \in \mathfrak{so}(3)^* \cong
\mathbb{R}^3$: the angular momentum (represented as a skew-symmetric
matrix);\quad $k \in \mathbb{R}^3$: the center of mass position (or
static moment);\quad $p \in \mathbb{R}^3$: the linear momentum;\quad
$E \in \mathbb{R}$: the energy of the system. The pairing between
the dual element $\mu$ and an element of the Lie algebra $\xi= (\Xi,
\vartheta, \nu, \varepsilon) \in \mathfrak{sgal}(3)$ is defined by
the Frobenius inner product
\begin{equation*}
\langle \mu, \xi \rangle = \mathrm{tr}(j^T \Xi) + k \cdot
\mathbf{\vartheta} + p \cdot \nu + E \varepsilon.
\end{equation*}

\end{proposition}

\begin{proposition}\textup{\cite{ma}}\label{th1}
 Let $\beta=\left(
                                   \begin{array}{cr}
                                     0 & -a \\
                                     a & 0 \\
                                   \end{array}
                                 \right)\in \mathfrak{so}(2), \; a\in \mathbb{R}$ an element of the Lie algebra
                                 $\mathfrak{so}(2)$.
                             Let  {$\Omega= \left\{ \left(
                                   \begin{array}{cr}
                                     0 & -a \\
                                     a & 0 \\
                                   \end{array}
                                 \right) \semicolon a > 0 \right\}$},and  {$\Omega^{*}=\left\{\xi=\left(
                                   \begin{array}{cr}
                                     0 & -x \\
                                     x & 0 \\
                                   \end{array}
                                 \right) \semicolon x > 0 \right\}$}
the Koszul dual cone. There exists a unique one-cocycle of the Lie
algebra $\Theta_{\beta}:\mathfrak{so}(2)\longrightarrow
\mathfrak{so}(2):X\mapsto\frac{1}{a^{2}}X$  which is linear for
$\beta$ fixed, symmetric and positive such that the distinguished
density function is given by
 {\begin{equation*}p(\beta,\xi)=\frac{\mathrm{e}^{-\langle\Theta_{\beta}^{-1}\left(\eta\right),\xi\rangle}}{\int_{\Omega^{*}}
\mathrm{e}^{-\langle\Theta_{\beta}^{-1}\left(\eta\right),\xi\rangle}
\mathrm{d}x},\qquad i.e.,\qquad
p(\beta,\xi)=\frac{\mathrm{e}^{-2ax}}{\int^{+\infty}_{0}
\mathrm{e}^{-2ax} \mathrm{d}x}\end{equation*}} the potential
function $\Phi$ and the dual potential function $\Psi$ satisfying
the Legendre equation
\begin{equation*}\Psi\left(\eta\right)=\langle\beta,\eta\rangle-\Phi\left(\beta\right)\end{equation*}
is given by
\begin{equation*}\Phi\left(\beta\right)=\log(2a),\qquad \Psi\left(\eta\right)=1-\log(2a).\end{equation*}
Such that
 {\begin{equation*}\frac{\partial\Psi\left(\eta\right)}{\partial
\eta}=\beta,\qquad and\qquad
\frac{\partial\Phi\left(\beta\right)}{\partial
\beta}=\eta.\end{equation*}} With $\eta=\frac{1}{a}\left(
                                                                                            \begin{array}{cr}
                                                                                              0 & -1 \\
                                                                                              1 & 0 \\
                                                                                            \end{array}
                                                                                          \right)$.\end{proposition}

In what follows we will consider the basis of the Lie algebra and
its dual basis as left invariants on the Lie algebra.

\section{ KV-Cohomology
and de Rham Cohomology on Lie Algebra
$\mathfrak{so}(2)$\label{sec:sec3}} In this section, we consider the
Cartan Schouten affine connection $\nabla_{X}Y=\frac{1}{2}\left[X,Y
\right] $ where $X,Y\in\mathfrak{X}(G)$, the $C^{\infty} $ -module
of left invariant vector fields on G=SO(2). Since the
Cartan-Schouten is left invariant, we identify the lie algebra
$\mathfrak{so}(2)$ with $\mathfrak{X}_{L}(G)$, the  space  of
left-invariant  vector fields. The Lie bracket on $\mathfrak{so}(2)$
the induced by $[u,v ]=[u^{l},v^{l}](e)$ for $u,\quad v\in
\mathfrak{so}(2)$ where $u^{l}$ is the left invariant vector field
induced by $u\in \mathfrak{so}(2)$. By the argument of dimension,
since  G   is  a one dimensional manifold,  the space
$\mathfrak{X}_{L}(G)$ has one dimension. The induced $KV$-algebra
$\mathcal{A}=\left(\mathfrak{so}(2),\nabla\right)$ is identify with
{$\mathfrak{so}(2)= \left\{ \left(
                                   \begin{array}{cr}
                                     0 & -a \\
                                     a & 0 \\
                                   \end{array}
                                 \right) \semicolon a \in \mathbb{R}
                                 \right\}$} and the $[.,.]$ on
                                 $\mathfrak{so}(2)$ is the
                                 commutator. The space $\mathfrak{so}(2)$ is endowed with canonical Frobenius
                                 inner product  $\langle
X,Y\rangle:=\frac{1}{2}tr\left(X^{T}Y\right)$. The Koszul a Koszul
dual cone are given by {$\Omega= \left\{ \left(
                                   \begin{array}{cr}
                                     0 & -a \\
                                     a & 0 \\
                                   \end{array}
                                 \right) \semicolon a > 0 \right\}$}
                                 and  {$\Omega^{*}=\left\{\left(
                                   \begin{array}{cr}
                                     0 & -x \\
                                     x & 0 \\
                                   \end{array}
                                 \right) \semicolon x > 0
                                 \right\}$}. The a above inner
                                 product is also the duality bracket
                                 since $\mathfrak{so}(2)^{*}$ is
                                 identify to $\mathfrak{so}(2)$.
                                 Where  $\mathfrak{so}(2)=\mathrm{vect}(e)$ as a vector space is generated by
                                 $e=\left(
                                                                                            \begin{array}{cc}
                                                                                              0 & -1 \\
                                                                                              1 & 0 \\
                                                                                            \end{array}
                                                                                          \right)$
                                                                                          as the left-invariant vector field.
\begin{proposition}\label{SO}
Let $a\in \mathbb{R}^{*}_{+}$,\quad $\Omega$, and $\Omega^{*}$ the
Koszul dual cone, with duality bracket define by the Frobenius inner
product $\langle-;-\rangle$. $\mathcal{A}$ is a KV-algebra over
$\mathbb{R}$ and given $\left(\mathfrak{so}(2),+,.\right)$ a
left-module over $\mathcal{A}$, the set of Jacobi elements is given
by $J\left(\mathfrak{so}(2)\right)=\mathfrak{so}(2)$.
\end{proposition}

\begin{proof}
Let $\alpha,\beta,\gamma\in \mathcal{A}$ such that $\beta=ae$,
$\gamma=be$, and $\alpha=ce$, with $a,b,c \in \mathbb{R}$. We have
the following relation
\begin{align*}
 KV(\alpha,\beta,\gamma) &= (\alpha\cdot\beta)\cdot\gamma-\alpha\cdot(\beta\cdot\gamma)-(\beta\cdot\alpha)\cdot\gamma+\beta\cdot(\alpha\cdot\gamma)= 0.
\end{align*}
So, $\mathcal{A}$ is a KV-algebra over $\mathbb{R}$. Furthermore,the
lie algebra $\mathfrak{so}(2)$ is a left KV- module over itself
(this is called the adjoint module). Indeed, as soon as a KV-algebra
structure is defined on a vector space A, this space automatically
becomes a left KV-module over itself, using its own internal product
as the action. So, for all $\rho=\lambda e\in \mathfrak{so}(2),\quad
\lambda\in \mathbb{R}$
\begin{align*}
    (\alpha,\beta,\rho) &=
    (\alpha\cdot\beta)\cdot\rho-\alpha\cdot(\beta\cdot\rho)=0.
\end{align*}
Thus we have the  following   Jacobi space
$J\left(\mathfrak{so}(2)\right)=\mathfrak{so}(2)$.

\end{proof}

In the following we take in \cite{ma}, the potential function
$\Phi(a)= \log(2a)$ as casimir function.
\begin{theorem}\label{SO2}
Let  $\Omega= \left\{ ae \semicolon a > 0 \right\}$, the  Koszul
                                 cone and, $\mathcal{A}$ is a  left
                                module on $\Omega$. Let $\mathcal{A}$ be the KV-algebra associated
with the left invariant flat connection on SO(2). The Cohomology
space of the KV Cohomology on SO(2) is given by
\begin{equation*}
\mathrm{H^{0}_{KV}(\mathcal{A},\mathfrak{so}(2))}\simeq\mathbb{R},\qquad
\mathrm{H^{1}_{KV}(\mathcal{A},\mathfrak{so}(2))}\simeq\mathbb{R},\qquad
\mathrm{H^{2}_{KV}(\mathcal{A},\mathfrak{so}(2))}=0.\end{equation*}
\end{theorem}

\begin{proof}
 Let $e $ a generator of
$\mathfrak{so}(2)$, \qquad $[e,e]=0$. Let $\beta=ae \in \mathfrak{so}(2)$. \\
We have, the the coboundary
$\delta_{KV}^{0}$ and $\delta_{KV}^{1}$ is define by\\
$\delta_{KV}^{0}: C^{0}\left(\mathcal{A},\mathfrak{so}(2)\right)
\longrightarrow
C^{1}\left(\mathcal{A},\mathfrak{so}(2)\right),\qquad
\delta_{KV}^{1}: C^{1}\left(\mathcal{A},\mathfrak{so}(2)\right)
\longrightarrow C^{2}\left(\mathcal{A},\mathfrak{so}(2)\right)$.

On  one-dimensional manifold, the spaces of forms are
\begin{eqnarray*}
C^{0}\left(\mathcal{A},\mathfrak{so}(2)\right)&=&
\mathfrak{so}(2),\quad
C^{1}\left(\mathcal{A},\mathfrak{so}(2)\right) =
 \mathrm{Hom_{\mathbb{R}}\left(\mathcal{A},\mathfrak{so}(2)\right)}\\
C^{2}\left(\mathcal{A},\mathfrak{so}(2)\right) &=&
\mathrm{Hom_{\mathbb{R}}\left(\mathcal{A}\otimes\mathcal{A},\mathfrak{so}(2)\right)}.
\end{eqnarray*}
Let us determine the operator $\delta_{KV}^{0}$:  for $\rho \in
\mathfrak{so}(2)$ the   differential  is  given by
\begin{eqnarray*}
\delta_{KV}^{0}: C^{0}\left(\mathcal{A},\mathfrak{so}(2)\right)
&\longrightarrow&
C^{1}\left(\mathcal{A},\mathfrak{so}(2)\right),\qquad
\rho\mapsto\delta_{KV}^{0}\rho
\end{eqnarray*}
and  \begin{eqnarray*} \delta_{KV}^{0}\rho (\beta) &=& -\beta.\rho
+\rho.\beta=0.
\end{eqnarray*}
 Using the Proposition \ref{th1}
                                we have
$\eta=\frac{1}{a}e$ the Koszul one-form\\
\begin{eqnarray*}
\delta_{KV}^{1}: C^{1}\left(\mathcal{A},\mathfrak{so}(2)\right)
&\longrightarrow&
C^{2}\left(\mathcal{A},\mathfrak{so}(2)\right),\qquad
\delta_{KV}^{1} \eta = 0.
\end{eqnarray*}
because, using the definition \ref{k1}, we  obtain
\begin{eqnarray*}
\delta_{KV}^{1}\eta (\beta,\beta) &=& -(\beta.\eta).
(\beta)-\eta(\beta). \beta=-\beta.\eta
(\beta)-\eta(\beta.\beta)-\eta(\beta). \beta=0.
\end{eqnarray*}
Because, both left actions of a vector field on a constant or a form
vanish $\beta.\eta (\beta)=0$ and the right module action is
identically zero
$\eta(\beta). \beta=0.$\\

Calculation of
$\mathrm{H^0_{KV}\left(\mathcal{A},\mathfrak{so}(2)\right)}$\\
\begin{eqnarray*}
\ker\delta_{KV}^{0}
&=&C^{0}\left(\mathcal{A},\mathfrak{so}(2)\right)= \mathfrak{so}(2)=
\mathfrak{so}(2)\cong \mathbb{R}.\end{eqnarray*} Since there are no
forms of degree $-1$
\begin{eqnarray*}
\mathrm{Im}  \delta_{KV}^{-1} = \{0\}
\end{eqnarray*}
Therefore,\quad
$\mathrm{H^0_{KV}\left(\mathcal{A},\mathfrak{so}(2)\right)}=
\frac{\ker\delta_{KV}^{0}}{\mathrm{Im\delta_{KV}}^{-1}}
\cong\mathbb{R}$.\\

Calculation of
$\mathrm{H^1_{KV}\left(\mathcal{A},\mathfrak{so}(2)\right)}$
\begin{eqnarray*}
\ker\delta_{KV}^{1} &=&
C^{1}\left(\mathcal{A},\mathfrak{so}(2)\right)=
\mathrm{Hom_{\mathbb{R}}\left(\mathcal{A},\mathfrak{so}(2)\right)}\cong\mathbb{R}.\end{eqnarray*}
We recall that $\delta^{0}_{KV}=0$, and  its image  is   given   by
$\mathrm{Im\delta^{0}_{KV}}=\left\{0\right\}$, then

\begin{eqnarray*}\mathrm{H^1_{KV}}(\mathcal{A},\mathfrak{so}(2))&=&
\frac{\ker\delta_{KV}^{1}}{\mathrm{Im} \delta_{KV}^{0}} \cong
\mathbb{R}.\end{eqnarray*} So we have,
\begin{eqnarray*}
\mathrm{H^0_{KV}\left(\mathcal{A},\mathfrak{so}(2)\right)} &=&
\mathbb{R},\qquad
\mathrm{H^1_{KV}\left(\mathcal{A},\mathfrak{so}(2)\right)} \cong
\mathbb{R}.
\end{eqnarray*}

\end{proof}

\begin{theorem}\label{SO21}
Let $\mathfrak{so}(2)$, be the Lie algebra. The Cohomology space of
De Rham on $\mathrm{SO(2)}\simeq S^1$  is given by
  \begin{equation*}
  \mathrm{ H^0(\mathfrak{so}(2),\mathbb{R})}\simeq \mathbb{R},
  \qquad \mathrm{H^1(\mathfrak{so}(2),\mathbb{R})}\simeq \mathbb{R}.\end{equation*}
\end{theorem}

\begin{proof}
The group SO(2) is isomorphic to the unit circle $S^1$ and has
dimension $1$.

\begin{eqnarray*}
\Omega^0(\mathfrak{so}(2)) &=& C^\infty(\mathrm{SO(2)})\cong\mathbb{R} \\
\Omega^1(\mathfrak{so}(2))&=& \{\eta = \diff\Phi \mid \Phi \in C^\infty(\mathrm{SO(2)})\}=\mathfrak{so}(2)^{*}\cong\mathbb{R} \\
\Omega^2(\mathfrak{so}(2))  &=& 0.
\end{eqnarray*}
The exterior differential operator is defined by
\begin{eqnarray*}
\diff^0: \Omega^0(\mathfrak{so}(2)) &\longrightarrow&
\Omega^1(\mathfrak{so}(2)), \qquad \Phi \mapsto \diff^0 \Phi.
\end{eqnarray*}
It is know that \begin{eqnarray*} \diff^k:
\Omega^k(\mathfrak{so}(2)) \longrightarrow
\Omega^{k+1}(\mathfrak{so}(2)) = 0,\qquad
 \textrm{for all} \quad   k \geq 1.
\end{eqnarray*}

Calculation of $\mathrm{H^0(\mathfrak{so}(2),\mathbb{R})}$
\begin{eqnarray*}
\ker\diff^0 &=& \{\Phi \in C^\infty(\mathrm{SO(2)}) \mid \Phi \text{
constant}\} \cong \mathbb{R}.\end{eqnarray*} Since there are no
forms of degree $-1$: $\mathrm{Im}\diff^{-1} = \{0\}$.\\
Therefore,\quad $\mathrm{H^0(\mathfrak{so}(2),\mathbb{R})}=
\frac{\ker \diff^0}{\mathrm{Im} \diff^{-1}} \cong \mathbb{R}$. It  follows  that $\mathrm{dim H^0(\mathfrak{so}(2),\mathbb{R})} = 1$.\\

Calculation of $\mathrm{H^1(\mathfrak{so}(2),\mathbb{R})}$.\\
We have $\eta=\frac{1}{a}e\in \mathfrak{so}(2)^{*}$ and $\beta=ae\in
\mathfrak{so}(2)$. Using the definition (\ref{rham}), we get
$\diff^1\eta(\beta,\beta) =
\beta(\eta(\beta))-\beta(\eta(\beta))-\eta\left([\beta,\beta]\right)=0$.\\

Since $\Omega^2(\mathfrak{so}(2)) = 0$, any one-form is
automatically closed, therefore, the differential operator is
defined by
\begin{eqnarray*}
\diff^1: \Omega^1\left(\mathfrak{so}(2)\right)  &\longrightarrow&
\Omega^2\left(\mathfrak{so}(2)\right), \qquad \eta \mapsto \diff^1
\eta = 0
\end{eqnarray*} and, we  obtain
\begin{eqnarray*}
\ker \diff^1= \Omega^1(\mathfrak{so}(2))\cong\mathbb{R}.
\end{eqnarray*}

and  $\diff^0=0$, so $\mathrm{Im} \diff^0=\{0\}$. Then $
\mathrm{H^1(\mathfrak{so}(2),\mathbb{R})}= \frac{\ker
\diff^1}{\mathrm{Im} \diff^0 }\cong \mathbb{R}.$

\end{proof}

\section{KV-Cohomology
and de Rham Cohomology on
\(\mathrm{H_{3}(\mathbb{R})}\)\label{sec:sec4}}

The group \(\mathrm{H_{3}(\mathbb{R})}\) is a connected, simply
connected, and contractible Lie group (as a manifold, it is
diffeomorphic to \(\mathbb{R}^{3}\)). Let $e_{1}=\left(
             \begin{array}{ccc}
               0 & 1 & 0 \\
               0 &  0 & 0 \\
               0 & 0 & 0 \\
             \end{array}
           \right),
e_{2}=\left(
             \begin{array}{ccc}
               0 & 0 & 0 \\
               0 &  0 & 1 \\
               0 & 0 & 0 \\
             \end{array}
           \right),e_{3}=\left(
             \begin{array}{ccc}
               0 & 0 & 1 \\
               0 &  0 & 0 \\
               0 & 0 & 0 \\
             \end{array}
           \right)$ a basis for the Lie algebra $\mathfrak{h}_{3}$,\quad $[e_{1},e_{2}]=e_{3},\quad[e_{2},e_{3}]=0,\quad[e_{3},e_{1}]=0.$
 We choose the standard
basis \(\mathcal{B}=\left\{\beta_{1},\beta_{2},\beta_{3}\right\}\)
corresponding to the matrices   \(\beta_{1}=ae_{1},\quad
\beta_{2}=ae_{2}, \quad \beta_{3}=ae_{3}\). The only non-zero
commutation relation is \([\beta_{1},\beta_{2}]=a\beta_{3}\) with
$a>0$.

\subsection{De Rham Cohomology on
\(\mathrm{H_{3}(\mathbb{R})}\)\label{ssec:math5}}

In this section, we characterize the De Rham Cohomology group on the
Heisenberg group using the following theorem
\begin{theorem}\label{dr}Let $\mathfrak{h}_{3}$ be the Heisenberg Lie algebra equipped with its
pre-Lie structure defined by $\beta_{1}\cdot \beta_{2}=a\beta_{3}$
with basis
$\mathcal{B}=\left\{\beta_{1},\beta_{2},\beta_{3}\right\}$. Let be
the connected, simply connected, and contractile Lie group
\(\mathrm{H_{3}(\mathbb{R})}\) which is diffeomorphic to
\(\mathbb{R}^{3}\). The Cohomology of the Lie algebra
\(\mathfrak{h}_{3}\) of left invariant forms, which reflects the
algebraic structure of the group, is given by
\begin{equation*}
 \mathrm{H^{0}(\mathfrak{h}_{3},\mathbb{R})}\cong
\mathbb{R},\qquad \mathrm{H^{1}(\mathfrak{h}_{3},\mathbb{R})}\cong
\mathbb{R}^{2},\qquad
\mathrm{H^{2}(\mathfrak{h}_{3},\mathbb{R})}\cong
\mathbb{R}^{2},\qquad
\mathrm{H^{3}(\mathfrak{h}_{3},\mathbb{R})}\cong
\mathbb{R}\end{equation*}
\end{theorem}

\begin{proof}
Let us calculate the differentials on the Heisenberg algebra
\(\mathfrak{h}_{3}\), first defining a basis for the Lie algebra and
its dual basis of left-invariant forms.  We choose the standard
basis \(\{\beta_{1},\beta_{2},\beta_{3}\}\) corresponding to the
matrices
\quad\(\beta_{1}=ae_{1},\quad\beta_{2}=ae_{2},\quad\beta_{3}=ae_{3}\).
The only non-zero commutation relation is
\([\beta_{1},\beta_{2}]=a\beta_{3}\).

By setting $\beta=\beta_{1}=ae_{1}$ we have $\Omega=\left\{ae_{1}
\semicolon a > 0 \right\}$. The characteristic function is given by
\begin{eqnarray*}
                                  \chi(\beta_{1}) &=&
                                  \int^{+\infty}_{0}
                                  \mathrm{e}^{-2ax}\diff
                                  x=\frac{1}{2a}\cdot
                                \end{eqnarray*}
                                Since the potential function is \begin{eqnarray*}
                                  \Phi(\beta_{1})&=&-\log\chi(\beta_{1})=\log(2a)=\log(2)+
                                  \log(a).
                                \end{eqnarray*}
Therefore \begin{eqnarray}\label{01}a^{2}&=&
\frac{1}{2}\mathrm{Tr}(\beta_{1}^{T}\beta_{1})=\langle\beta_{1},\beta_{1}\rangle.\end{eqnarray}
So using (\ref{01}) we obtain the derivative of $\Phi$ by
\begin{eqnarray*}\frac{\partial\Phi\left(\beta_{1}\right)}{\partial\beta_{1}}&=&
\frac{1}{2}\frac{\beta_{1}+\beta_{1}}{\langle\beta_{1},\beta_{1}\rangle}=\frac{\beta_{1}}{\langle\beta_{1},\beta_{1}\rangle}
=\frac{1}{a^{2}}\beta_{1}=\frac{1}{a}e_{1}
\end{eqnarray*}
Proceeding in the same way for $\beta=\beta_{2}$  and
$\beta=\beta_{3}$ we construct the basis
$\mathcal{B}^{*}=\left\{\eta_{1} ,\eta_{2} ,\eta_{3} \right\}$, with
$\left\{\eta_1=\frac{1}{a}e_{1},\quad \eta_2=\frac{1}{a}e_{2},\quad
\eta_3=\frac{1}{a}e_{3}\right\}$ the dual basis of left-invariant
one-forms, such that\\ $\eta_{1} (\beta_{1})=1,\quad\eta_{2}
(\beta_{2})=1,\quad\eta_{3} (\beta_{3})=1$. Using the differential
in definition \ref{rham}, we have $ \eta_{1} (\beta_{3}) = 0,
\quad\eta_{2} (\beta_{3}) = 0,\quad \eta_{1} (\beta_{2}) = 0,\quad
\eta_{2} (\beta_{1}) = 0,\quad \eta_{3} (\beta_{1}) =
0,\quad\eta_{3} (\beta_{2}) = 0.$ We have the following equation
\begin{eqnarray}\label{eq0}
  \diff^{1}\eta_{1}(\beta_{1},\beta_{2})&=&-\eta_{1}
([\beta_{1},\beta_{2}])=-\eta_{1} (a\beta_{3})=0 \nonumber\\
  \diff^{1}\eta_{2} (\beta_{1},\beta_{2})&=&-\eta_{2}
([\beta_{1},\beta_{2}])=-\eta_{2} (a\beta_{3})=0\\
 \diff^{1}\eta_{3} (\beta_{1},\beta_{2})&=&-\eta_{3}
([\beta_{1},\beta_{2}])=-\eta_{3} (a\beta_{3})=-a\nonumber
\end{eqnarray}
\begin{eqnarray}\label{eq1}
     \diff^{1}\eta_{1}(\beta_{2},\beta_{3})&=&\beta_{2}\eta_{1} (\beta_{3})-\beta_{3}\eta_{1} (\beta_{2})-\eta_{1}
([\beta_{2},\beta_{3}])=0 \nonumber\\
    \diff^{1}\eta_{2}(\beta_{2},\beta_{3})&=&\beta_{2}\eta_{2} (\beta_{3})-\beta_{3}\eta_{2} (\beta_{2})-\eta_{2}
([\beta_{2},\beta_{3}])=0 \\
    \diff^{1}\eta_{3}(\beta_{2},\beta_{3})&=&\beta_{2}\eta_{3} (\beta_{3})-\beta_{3}\eta_{3} (\beta_{2})-\eta_{3}
([\beta_{2},\beta_{3}])=0\nonumber
\end{eqnarray}

\begin{eqnarray}\label{eq2}
  \diff^{1}\eta_{1}(\beta_{3},\beta_{1})&=&0,\qquad \diff^{1}\eta_{2}(\beta_{3},\beta_{1})=0,\qquad \diff^{1}\eta_{3}(\beta_{2},\beta_{3})=0.
\end{eqnarray}
Let us now calculate the invariant the De Rham Cohomology.  We are
looking for closed forms ($\diff\eta=0$) modulo exact forms
($\eta =\diff\omega $)\\
Let $\Phi\in C^{\infty}(\mathrm{H_{3}(\mathbb{R})})$ we have
$\diff^{0} \Phi(a)=0.$ so, we have\begin{eqnarray*}
  \ker  \diff^{0}&=&\mathbb{R},\qquad \mathrm{Im}  \diff^{-1}=\{0\},\qquad \mathrm{H^{0}(\mathfrak{h}_{3},\mathbb{R})}\cong
\mathbb{R}
\end{eqnarray*}
the invariant zero-forms are constant.\\
First degree: Let $\eta=x\eta_{1}+y\eta_{2}+z\eta_{3},\quad
x,y,z\in\mathbb{R}$ we have
$\diff^{1}\eta=x\diff^{1}\eta_{1}+y\diff^{1}\eta_{2}+z\diff^{1}\eta_{3}$

using (\ref{eq0}),\quad(\ref{eq1}),\quad and (\ref{eq2}), we have
the following equation
\begin{equation*}
         \diff^{1}\eta(\beta_{1},\beta_{2})= -az,\qquad  \diff^{1}\eta(\beta_{2},\beta_{3})= 0,\qquad \diff^{1}\eta(\beta_{3},\beta_{1})=0
\end{equation*}

\begin{equation*}
    \diff^{1}\eta=0\Longleftrightarrow z=0
\end{equation*}
Consequently $\eta=x\eta_{1}+y\eta_{2},\qquad x,y\in\mathbb{R}.$

$\ker\diff^{1}= \left\{\eta | \eta=x\eta_{1}+y\eta_{2},\quad
x,y\in\mathbb{R}\right\}$. Closed forms are generated by $\{\eta_{1}
,\eta_{2} \}$,\\ $\mathrm{dim}\ker\diff^{1}=2$. and
$\mathrm{Im}\diff^{0}=\{0\}$. Which  implies   that
$\mathrm{H^{1}}(\mathfrak{h}_{3},\mathbb{R})=\text{Vect}(\eta_{1}
,\eta_{2}
)\cong \mathbb{R}^{2}$.\\
Second degree: We calculate the differentials of the 2-forms: Using
\ref{eq0},\ref{eq1},and \ref{eq2}, we obtain
\begin{equation*}
         \diff^{2}\eta_{1}= 0,\qquad \diff^{2}\eta_{2}= 0,\qquad  \diff^{2}\eta_{3}=-a\eta_{1} \land \eta_{2}.
\end{equation*}
Using the differential in Definition \ref{rham}, we get
\begin{eqnarray*} \diff^{2}(\eta_{1} \land \eta_{2}
)&=&\diff^{2}\eta_{1} \land \eta_{2} -\eta_{1}
\land\diff^{2}\eta_{2}
=0\\
\diff^{2}(\eta_{1} \land \eta_{3} )&=&\diff^{2}\eta_{1} \land
\eta_{3} -\eta_{1} \land \diff^{2}\eta_{3} =-\eta_{1} \land
(-\eta_{1} \land
\eta_{2} )=0\\
\diff^{2}(\eta_{2} \land \eta_{3} )&=&\diff^{2}\eta_{2} \land
\eta_{3} -\eta_{2} \land \diff^{2}\eta_{3} =-\eta_{2} \land
(-\eta_{1} \land \eta_{2} )=0.\end{eqnarray*}
 All two-forms are closed.
 $\;\mathrm{dim}\ker\diff^{2}=3$.  However,
$a\eta_{1} \land \eta_{2} =-\diff^{2}\eta_{3}$ is exact, and
$\mathrm{dim
 Im}\diff^{1}=1$. Thus, $\mathrm{H^{2}(\mathfrak{h}_{3},\mathbb{R})}=\text{Vect}(\eta_{1} \land \eta_{2}
,\eta_{2} \land \eta_{3} )\cong \mathbb{R}^{2}$.\\

Third degree: The three-form of volume $\eta_{1} \land \eta_{2}
\land \eta_{3} $ is closed and inexact,
$\mathrm{dim}\ker\diff^{3}=1$, and $\mathrm{Im}\diff^{2}=\{0\}$.
Therefore $\mathrm{H^{3}(\mathfrak{h}_{3},\mathbb{R})\cong
\mathbb{R}}$.
\end{proof}

\subsection{KV-Cohomology  on $\mathrm{H_{3}(\mathbb{R})}$\label{ssec:math6}}
Let $\mathcal{A}=(\mathfrak{h}_{3},.)$ the KV-algebra over
$\mathrm{H_{3}}(\mathbb{R})$ where $(.)$ is the pre Lie structure.
In this section, we characterize the KV-Cohomology
$\mathrm{H_{3}(\mathbb{R})}$ group on the Heisenberg group. We have
the  following theorem
\begin{theorem}\label{H(3)}Let $\mathfrak{h}_{3}$ be the Heisenberg Lie algebra equipped with its
pre-Lie structure defined by $\beta_{1}\cdot \beta_{2}=a\beta_{3}$
with basis
$\mathcal{B}=\left\{\beta_{1},\beta_{2},\beta_{3}\right\}$. Let
$\Omega=\left\{\beta,\;\beta\in\mathcal{B}\right\}$
                                a nondegenerate open convex cone containing  no affine line.
                                The Koszul-Vinberg Cohomology
groups with coefficients in the adjoint module are given by
\begin{equation*}
\mathrm{H_{KV}^{0}(\mathcal{A},\mathfrak{h}_{3})}\cong
\mathbb{R},\qquad\mathrm{H_{KV}^{1}(\mathcal{A},\mathfrak{h}_{3})}=\{0\},\;
\mathrm{H_{KV}^{2}(\mathcal{A},\mathfrak{h}_{3})}\cong\mathbb{R}^{21}.\end{equation*}
\end{theorem}

\begin{proof}
Let \(\mathcal{B}=\left\{\beta_{1},\beta_{2},\beta_{3}\right\}\) and
$\mathcal{B}^{*}=\left\{\eta_{1} ,\eta_{2} ,\eta_{3} \right\}$,
with\newline $\left\{\eta_{1} ,\eta_{2} ,\eta_{3} \right\}$, the
dual basis of left-invariant one-forms, such that\\ $\eta_{1}
(\beta_{1})=1,\;\eta_{2} (\beta_{2})=1,\;\eta_{3} (\beta_{3})=1$,
and satisfy the following relation
\begin{equation} \eta_i(\beta_j) = \delta^i_j.\end{equation}

Let us determine the operator $\delta_{KV}^{0}$
\begin{eqnarray*}
\delta_{KV}^{0}: C^{0}(\mathcal{A},\mathfrak{h}_{3})
&\longrightarrow& C^{1}(\mathcal{A},\mathfrak{h}_{3})
\end{eqnarray*}
with $C^{0}(\mathcal{A},\mathfrak{h}_{3})=\mathfrak{h}_{3}$ and
$C^{1}(\mathcal{A},\mathfrak{h}_{3})
=\mathrm{Hom_{\mathbb{R}}\left(\mathcal{A},\mathfrak{h}_{3}\right)}$.
For all $\phi\in \mathfrak{h}_{3}$ we obtain
\begin{eqnarray*}
\delta_{KV}^{0}\phi (\beta_{1}) &=& -\beta_{1}\phi+ \phi \beta_{1}=0.\\
\delta_{KV}^{0}\phi (\beta_{2}) &=& -\beta_{2}\phi+ \phi \beta_{2}=0.\\
\delta_{KV}^{0}\phi (\beta_{3}) &=& -\beta_{3}\phi+ \phi\beta_{3}=0.
\end{eqnarray*}
It follows that
\begin{eqnarray*}
\delta_{KV}^{1}: C^{1}(\mathcal{A},\mathfrak{h}_{3})
&\longrightarrow& C^{2}(\mathcal{A},\mathfrak{h}_{3})\end{eqnarray*}

with $C^{1}(\mathcal{A},\mathfrak{h}_{3})=
\mathrm{Hom_{\mathbb{R}}}\left(\mathcal{A},\mathfrak{h}_{3}\right)$,\quad
$C^{2}(\mathcal{A},\mathfrak{h}_{3})=\mathrm{Hom_{\mathbb{R}}}
\left(\mathcal{A}\otimes\mathcal{A},\mathfrak{h}_{3}\right)$,\\ and
$C^{3}(\mathcal{A},\mathfrak{h}_{3})=\mathrm{Hom_{\mathbb{R}}}
\left(\mathcal{A}\otimes\mathcal{A}\otimes\mathcal{A},\mathfrak{h}_{3}\right)$.
Hence
\begin{eqnarray*}
\delta_{KV}^{1}\eta_{1} (\beta_{1},\beta_{2}) &=&
-\left(\beta_{1}\eta_{1}\right)\left(\beta_{2}\right)-
\left(\eta_{1}(\beta_{1})\right)\beta_{2}=-2\beta_{2}\\
\delta_{KV}^{1}\eta_{2} (\beta_{1},\beta_{2}) &=&
-\left(\beta_{1}\eta_{2}\right)\left(\beta_{2}\right)-
\left(\eta_{2}(\beta_{1})\right)\beta_{2}=0\\
\delta_{KV}^{1}\eta_{3} (\beta_{1},\beta_{2}) &=&
-\left(\beta_{1}\eta_{3}\right)\left(\beta_{2}\right)-
\left(\eta_{3}(\beta_{1})\right)\beta_{2}=0\\
 \delta_{KV}^{1}\eta_{1}
(\beta_{2},\beta_{3}) &=&
-\left(\beta_{2}\eta_{1}\right)\left(\beta_{3}\right)-
\left(\eta_{1}(\beta_{2})\right)\beta_{3}=0\\
\delta_{KV}^{1}\eta_{2} (\beta_{2},\beta_{3}) &=&
-\left(\beta_{2}\eta_{2}\right)\left(\beta_{3}\right)-
\left(\eta_{2}(\beta_{2})\right)\beta_{3}=-2\beta_{3}\\
\delta_{KV}^{1}\eta_{3} (\beta_{2},\beta_{3}) &=&
-\left(\beta_{2}\eta_{3}\right)\left(\beta_{3}\right)-
\left(\eta_{3}(\beta_{2})\right)\beta_{3}=0\\
\delta_{KV}^{1}\eta_{1} (\beta_{3},\beta_{1}) &=&
-\left(\beta_{3}\eta_{1}\right)\left(\beta_{1}\right)-
\left(\eta_{1}(\beta_{3})\right)\beta_{1}=0\\
\delta_{KV}^{1}\eta_{2} (\beta_{3},\beta_{1}) &=&
-\left(\beta_{3}\eta_{2}\right)\left(\beta_{1}\right)-
\left(\eta_{2}(\beta_{3})\right)\beta_{1}=0\\
\delta_{KV}^{1}\eta_{3} (\beta_{3},\beta_{1}) &=&
-\left(\beta_{3}\eta_{3}\right)\left(\beta_{1}\right)-
\left(\eta_{3}(\beta_{3})\right)\beta_{1}=-2\beta_{1}
\end{eqnarray*}

\begin{eqnarray*}
\delta_{KV}^{2}: C^{2}(\mathcal{A},\mathfrak{h}_{3})
&\longrightarrow& C^{3}(\mathcal{A},\mathfrak{h}_{3})
\end{eqnarray*}

\begin{eqnarray*}
\delta_{KV}^{2}(\delta_{KV}^{1}\eta_{1})
(\beta_{1},\beta_{2},\beta_{3}) &=&
-\left(\beta_{1}\delta_{KV}^{1}\eta_{1}\right)\left(\beta_{2},\beta_{3}\right)-
\delta_{KV}^{1}\eta_{1}\left(\beta_{2},\beta_{1}\right).\beta_{3}\\
&&
+\left(\beta_{2}\delta_{KV}^{1}\eta_{1}\right)\left(\beta_{1},\beta_{3}\right)+
\delta_{KV}^{1}\eta_{1}\left(\beta_{1},\beta_{2}\right).\beta_{3}=0.
\end{eqnarray*}
Thus, yields
\begin{eqnarray*}
\delta_{KV}^{2}(\delta_{KV}^{1}\eta_{1})
(\beta_{1},\beta_{2},\beta_{3}) &=&  2\beta_{2}.\beta_{3} +
2\beta_{2}.\beta_{3}=0.
\end{eqnarray*}
Furthermore,
\begin{eqnarray*}
\delta_{KV}^{2}(\delta_{KV}^{1}\eta_{2})
(\beta_{1},\beta_{2},\beta_{3}) &=&
-\left(\beta_{1}\delta_{KV}^{1}\eta_{2}\right)\left(\beta_{2},\beta_{3}\right)-
\delta_{KV}^{1}\eta_{2}\left(\beta_{2},\beta_{1}\right).\beta_{3}\\
&&
+\left(\beta_{2}\delta_{KV}^{1}\eta_{2}\right)\left(\beta_{1},\beta_{3}\right)+
\delta_{KV}^{1}\eta_{2}\left(\beta_{1},\beta_{2}\right).\beta_{3}=0.
\end{eqnarray*}
So, we have
\begin{eqnarray*}
\delta_{KV}^{2}(\delta_{KV}^{1}\eta_{2})
(\beta_{1},\beta_{2},\beta_{3}) &=&  2\beta_{1}.\beta_{3} =0.
\end{eqnarray*}
However, we will have
\begin{eqnarray*}
\delta_{KV}^{2}(\delta_{KV}^{1}\eta_{3})
(\beta_{1},\beta_{2},\beta_{3}) &=&
-\left(\beta_{1}\delta_{KV}^{1}\eta_{3}\right)\left(\beta_{2},\beta_{3}\right)-
\delta_{KV}^{1}\eta_{3}\left(\beta_{2},\beta_{1}\right).\beta_{3}\\
&&
+\left(\beta_{2}\delta_{KV}^{1}\eta_{3}\right)\left(\beta_{1},\beta_{3}\right)+
\delta_{KV}^{1}\eta_{3}\left(\beta_{1},\beta_{2}\right).\beta_{3}=0.
\end{eqnarray*}
So, the previous equations becomes
\begin{eqnarray*}
\delta_{KV}^{2}(\delta_{KV}^{1}\eta_{3})
(\beta_{1},\beta_{2},\beta_{3}) &=&  -2\beta_{2}.\beta_{1} =0.
\end{eqnarray*}

We find  that \begin{eqnarray*}
\delta_{KV}^{2}(\delta_{KV}^{1}\eta_{1})&=&0,\quad
\delta_{KV}^{2}(\delta_{KV}^{1}\eta_{2})=0,\quad
\delta_{KV}^{2}(\delta_{KV}^{1}\eta_{3})=0.
\end{eqnarray*}
then,  we  deduce that
\begin{eqnarray*}
\delta_{KV}^{2}\circ\delta_{KV}^{1}=0.
\end{eqnarray*}

Calculation of $\mathrm{H^0_{KV}(\mathcal{A},\mathfrak{h}_{3})}$.
\begin{eqnarray*}
\ker\delta_{KV}^{0} &=& \{\Phi \in
C^\infty(\mathrm{H_{3}(\mathbb{R})}) \mid \Phi \text{ constant}\}
\cong \mathbb{R}.\end{eqnarray*} Since there are no forms of degree
$-1$: \quad$ \mathrm{Im}\delta_{KV}^{-1} = \{0\}.$\\
 Therefore
\begin{eqnarray*}\mathrm{H^0_{KV}(\mathcal{A},\mathfrak{h}_{3})}&=& \frac{\ker\delta_{KV}^{0}}{\mathrm{Im}\delta_{KV}^{-1}} \cong
\mathbb{R}.\end{eqnarray*}

Calculation of $\mathrm{H^1_{KV}(\mathcal{A},\mathfrak{h}_{3})}$\\
For all $\eta \in \mathfrak{h}_{3}^{*}$, taking
$\eta=x\eta_{1}+y\eta_{2}+z\eta_{3}$ yields
\begin{eqnarray*}
\delta_{KV}^{1}\eta(\beta_{1},\beta_{2})=x\delta_{KV}^{1}\eta_{1}(\beta_{1},\beta_{2})+y\delta_{KV}^{1}\eta_{2}(\beta_{1},\beta_{2})
+z\delta_{KV}^{1}\eta_{3}(\beta_{1},\beta_{2})\\
\delta_{KV}^{1}\eta(\beta_{2},\beta_{3})=x\delta_{KV}^{1}\eta_{1}(\beta_{2},\beta_{3})
+y\delta_{KV}^{1}\eta_{2}(\beta_{2},\beta_{3})
+z\delta_{KV}^{1}\eta_{3}(\beta_{2},\beta_{3})\\
\delta_{KV}^{1}\eta(\beta_{3},\beta_{1})=x\delta_{KV}^{1}\eta_{1}(\beta_{3},\beta_{1})
+y\delta_{KV}^{1}\eta_{2}(\beta_{3},\beta_{1})
+z\delta_{KV}^{1}\eta_{3}(\beta_{3},\beta_{1}).
\end{eqnarray*}
However, we  obtain
\begin{eqnarray*}
\delta_{KV}^{1}\eta (\beta_{1},\beta_{2})&=&x\delta_{KV}^{1}\eta_{1}
(\beta_{1},\beta_{2}),\quad
\delta_{KV}^{1}\eta (\beta_{2},\beta_{3})=y\delta_{KV}^{1}\eta_{2} (\beta_{2},\beta_{3})\\
\delta_{KV}^{1}\eta(\beta_{3},\beta_{1}) &=&z\delta_{KV}^{1}\eta_{1}
(\beta_{3},\beta_{1}).
\end{eqnarray*}
Therefore,  we conclude that
\begin{eqnarray*}
\delta_{KV}^{1}\eta &=&0\quad \Leftrightarrow\quad x=y=z=0
\end{eqnarray*}
\begin{eqnarray*}
\ker\delta_{KV}^{1} &=& \left\{0\right\},\quad and\quad \mathrm{Im}
 \delta_{KV}^{0} = \left\{0\right\}.
\end{eqnarray*}
So,  we have
\begin{eqnarray*}\mathrm{H^1_{KV}(\mathcal{A},\mathfrak{h}_{3})}&=& \frac{\ker\delta_{KV}^{1}}{\mathrm{Im}\delta_{KV}^{0}}
=\left\{0\right\}\end{eqnarray*}
Calculation of $\mathrm{H^2_{KV}(\mathcal{A},\mathfrak{h}_{3})}$\\
All element of the image is written as
sequence\\$\delta_{KV}^{1}\eta=x\delta_{KV}^{1}\eta_{1}+y\delta_{KV}^{1}\eta_{2}+z\delta_{KV}^{1}\eta_{3}.$
It follows that
\begin{eqnarray*}
\delta_{KV}^{1}\eta (\beta_{1},\beta_{2})
&=&x\delta_{KV}^{1}\eta_{1} (\beta_{1},\beta_{2}),\quad
\delta_{KV}^{1}\eta (\beta_{2},\beta_{3})
=y\delta_{KV}^{1}\eta_{2} (\beta_{2},\beta_{3})\\
\delta_{KV}^{1}\eta(\beta_{3},\beta_{1}) &=&z\delta_{KV}^{1}\eta_{1}
(\beta_{3},\beta_{1})
\end{eqnarray*}
a basis of the image is
$\left\{\delta_{KV}^{1}\eta_{1},\delta_{KV}^{1}\eta_{2},\delta_{KV}^{1}\eta_{3}\right\}$
and $\mathrm{dim Im}\delta_{KV}^{0}=3$.

\begin{eqnarray*}
\mathrm{Im} \delta_{KV}^{0} =
\mathrm{span}\left\{\delta_{KV}^{1}\eta_{1},\quad
\delta_{KV}^{1}\eta_{2},\quad \delta_{KV}^{1}\eta_{3}\right\}
\end{eqnarray*}
By  the  same token,
\begin{eqnarray*}
\ker\delta_{KV}^{2} &=&C^{2}(\mathcal{A},\mathfrak{h}_{3})
=\mathrm{Hom(\mathcal{A}\otimes\mathcal{A},\mathfrak{h}_{3})}.\end{eqnarray*}
We know that $\mathrm{dim \mathfrak{h}_{3}}=3$, so $3\times 3 =9$
ordered pairs. Each image lives in a three-dimensional space. Thus
\begin{eqnarray*} \mathrm{dim}
C^{2}(\mathcal{A},\mathfrak{h}_{3})=27,\quad \mathrm{dim}
\ker\delta_{KV}^{2}=27-3=24\end{eqnarray*}

\begin{eqnarray*}
\mathrm{dim H^2_{KV}(\mathcal{A},\mathfrak{h}_{3})}&=& \mathrm{dim}
\ker\delta_{KV}^{2}-\mathrm{dim \quad Im }\delta_{KV}^{1}=21
\end{eqnarray*}

Ultimately, we   obtain
\begin{eqnarray*}
\mathrm{H^0_{KV}(\mathcal{A},\mathfrak{h}_{3})} &=& \mathbb{R},\quad
\mathrm{H^1_{KV}(\mathcal{A},\mathfrak{h}_{3})} = 0, \quad
\mathrm{H^2_{KV}(\mathcal{A},\mathfrak{h}_{3})} \simeq
\mathbb{R}^{21}.
\end{eqnarray*}

\end{proof}

\section{ The de Rham Cohomology on
Special Galilei Group SGal(3)\label{sec:sec4}} In this section, we
characterize the The de Rham Cohomology on Special Galilei group
SGal(3). We have the  following theorem
\begin{theorem}\label{SO(31)}
Let $\mathfrak{so}(3)$ be the Lie algebra. The Cohomology of the Lie
algebra \(\mathfrak{so}(3)\) of left invariant forms, which reflects
the algebraic structure of the group, is given by
\begin{equation*}
\mathrm{H^{0}(\mathfrak{so}(3),\mathbb{R})}=\mathbb{R},\quad\mathrm{H^{1}(\mathfrak{so}(3),\mathbb{R})}=0,\quad
\mathrm{H^{2}(\mathfrak{so}(3),\mathbb{R})}=0,\quad\mathrm{H^{3}(\mathfrak{so}(3),\mathbb{R})}=\mathbb{R}.\end{equation*}
\end{theorem}

\begin{proof}Let $e_{1}=\left(
             \begin{array}{ccc}
               0 & 0 & 0 \\
               0 &  0 & -1 \\
               0 & 1 & 0 \\
             \end{array}
           \right),\quad
e_{2}=\left(
             \begin{array}{ccc}
               0 & 0 & 1 \\
               0 &  0 & 0 \\
               -1 & 0 & 0 \\
             \end{array}
           \right),\quad
           e_{3}=\left(
             \begin{array}{ccc}
               0 & -1 & 0 \\
               1 &  0 & 0 \\
               0 & 0 & 0 \\
             \end{array}
           \right)$ a basis Lie algebra $\mathfrak{so}(3)$,\quad $[e_{1},e_{2}]=e_{3},\quad[e_{2},e_{3}]=e_{1},
           \quad[e_{3},e_{1}]=e_{2}$.
 Let $\{\eta_1, \eta_2, \eta_3\}$ be the dual
left-invariant one-forms of the basis $\{e_1, e_2, e_3\}$
\begin{equation*} \eta_i(e_j) = \delta^i_j \end{equation*}

By setting  $\omega = \eta_1$, and using \ref{f1}, we obtain
\begin{equation*}
\diff\omega(e_1, e_2) =\diff\eta_1(e_1, e_2) = e_1(\eta_1(e_2)) -
e_2(\eta_1(e_1)) - \eta_1([e_1, e_2])
\end{equation*}
By definition of dual forms $ \eta_1(e_2) = 0.$ Consequently, $
e_1(\eta_1(e_2)) = e_1(0) = 0$. Similarly,\quad $ \eta_1(e_1) =1.$
Thus, $e_2(\eta_1(e_1)) = e_2(1) = 0.$ By the commutation relations
$[e_1, e_2] = e_3$, it   follows  that $ \eta_1([e_1, e_2]) =
\eta_1(e_3) = 0$. We   obtain
\begin{equation*} \diff\eta_1(e_1, e_2) = 0.\end{equation*} Hence,
\begin{equation*} \diff\eta_1(e_2, e_3) = e_2(\eta_1(e_3)) -
e_3(\eta_1(e_2)) - \eta_1([e_2, e_3])
\end{equation*}
By definition of dual forms $ \eta_1(e_3) = 0. $ Therefore $
e_2(\eta_1(e_3)) = e_2(0) = 0$. Likewise,\quad $ \eta_1(e_2) =0.$
Consequently, $e_3(\eta_1(e_2)) = e_3(0) = 0.$ According to the
commutation relations  $[e_2, e_3] = e_1$, we  deduce that $
\eta_1([e_2, e_3]) =
\eta_1(e_1) = 1$. We  obtain $ \diff\eta_1(e_2, e_3) =  -1$.\\
Therefore, \begin{equation*} \diff\eta_1(e_3, e_1) =
e_3(\eta_1(e_1)) - e_1(\eta_1(e_3)) - \eta_1([e_3, e_1]).
\end{equation*}
By definition of dual forms $ \eta_1(e_3) = 0,$ which  yields $
e_3(\eta_1(e_1)) = e_3(1) = 0$. Moreover,   $ \eta_1(e_3) =0$
implying  that, $e_1(\eta_1(e_3)) = e_1(0) = 0$. By the commutation
relations  $[e_3, e_1] = e_2$, hence $ \eta_1([e_3, e_1]) =
\eta_1(e_2) = 0$. We obtain,
\begin{equation*} \diff\eta_1(e_3, e_1) = 0.\end{equation*}
Furthermore,
\begin{equation*}
\diff\eta_2(e_1, e_2) = e_1(\eta_2(e_2)) - e_2(\eta_2(e_1)) -
\eta_2([e_1, e_2]).
\end{equation*}
By definition of dual forms  $ \eta_2(e_2) = 1. $ Thus, $
e_1(\eta_2(e_2)) = e_1(1) = 0$. Likewise,\quad $ \eta_2(e_1) =0.$
So, $e_2(\eta_2(e_1)) = e_2(0) = 0.$ According to the commutation
relations  $[e_1, e_2] = e_3$, hence  $ \eta_2([e_1, e_2]) =
\eta_2(e_3) = 0$. We obtain\quad $ \diff\eta_2(e_1, e_2) = 0$.
Similarly,  we  have
\begin{equation*} \diff\eta_2(e_2, e_3) = e_2(\eta_2(e_3)) -
e_3(\eta_2(e_2)) - \eta_2([e_2, e_3]).
\end{equation*}
By definition of dual forms $ \eta_1(e_3) = 0.$ Consequently, $
e_2(\eta_2(e_3)) = e_2(0) = 0$. Likewise,\quad $ \eta_2(e_2) =1,$
which leads  to $e_3(\eta_2(e_2)) = e_3(1) = 0.$  By   the
commutation relations $[e_2, e_3] = e_1$, it  follows  that
$\eta_2([e_2, e_3]) = \eta_2(e_1) = 0$. We conclude that
\begin{equation*} \diff\eta_2(e_2, e_3) = 0.\end{equation*}

In  a Similarly way,  \begin{equation*} \diff\eta_2(e_3, e_1) =
e_3(\eta_2(e_1)) - e_1(\eta_2(e_3)) - \eta_2([e_3, e_1]).
\end{equation*}
By definition of dual forms $ \eta_2(e_3) = 0$ which gives  $
e_3(\eta_2(e_1)) = e_3(1) = 0$. As a result,\quad $ \eta_2(e_3) =0,$
and  thus  $e_1(\eta_2(e_3)) = e_1(0) = 0.$ According to the
commutation relations $[e_3, e_1] = e_2$, we infer that $
\eta_2([e_3, e_1]) = \eta_2(e_2) = 1$. We obtain $ \diff\eta_2(e_3,
e_1) = -1$. Moreover,
\begin{equation*}
\diff\eta_3(e_1, e_2) = e_1(\eta_3(e_2)) - e_2(\eta_3(e_1)) -
\eta_3([e_1, e_2]).
\end{equation*}
By definition of dual forms $ \eta_2(e_2) = 1. $ Therefore, $
e_1(\eta_3(e_2)) = e_1(1) = 0$. Likewise,\newline $ \eta_3(e_1) =0,$
yielding $e_2(\eta_3(e_1)) = e_2(0) = 0.$ By the commutation
relations
\newline $[e_1, e_2] = e_3$, hence $ \eta_3([e_1, e_2]) =
\eta_3(e_3) = 1$. We obtain,  $ \diff\eta_3(e_1, e_2) = 0 - 0 - 1=
-$1. Likewise
\begin{equation*} \diff\eta_3(e_2, e_3) = e_2(\eta_3(e_3)) -
e_3(\eta_3(e_2)) - \eta_3([e_2, e_3]).
\end{equation*}
By definition of dual forms $ \eta_3(e_3) = 1.$ Consequently, $
e_2(\eta_3(e_3)) = e_2(0) = 0$. Similarly,\newline $ \eta_3(e_2)
=1.$ Thus, $e_3(\eta_3(e_2)) = e_3(1) = 0.$ According to the
commutation relations  $[e_2, e_3] = e_1$, it  follows  that $
\eta_3([e_2, e_3]) = \eta_3(e_1) = 0$. We obtain $ \diff\eta_3(e_2,
e_3) = 0$. Similarly
\begin{equation*} \diff\eta_3(e_3, e_1) = e_3(\eta_3(e_1)) -
e_1(\eta_3(e_3)) - \eta_3([e_3, e_1]).
\end{equation*}
By definition of dual forms $ \eta_3(e_3) = 1. $ Hence, $
e_3(\eta_3(e_1)) = e_3(0) = 0$.\\ So,$ \eta_3(e_3) =0,$ and
$e_1(\eta_3(e_3)) = e_1(1) = 0.$ By the commutation relations $[e_3,
e_1] = e_2$, we deduce  that $ \eta_3([e_3, e_1]) = \eta_3(e_2) =
0$. We obtain,\begin{equation*}\diff\eta_3(e_3, e_1) =
0.\end{equation*} However, it  is known that
\begin{eqnarray*}
 (\eta_2 \wedge \eta_3)(e_1, e_2) &=&
\eta_2(e_1)\eta_3(e_2) - \eta_2(e_2)\eta_3(e_1)
\end{eqnarray*} which  leads  to
\begin{eqnarray*}
  (-\eta_2 \wedge \eta_3)(e_2, e_3) &=& -(\eta_2(e_2)\eta_3(e_3) - \eta_2(e_3)\eta_3(e_2))  = -1 \\
  (-\eta_3 \wedge \eta_1)(e_3, e_1) &=& -(\eta_3(e_3)\eta_1(e_1) - \eta_3(e_1)\eta_1(e_3))  = -1 \\
  (-\eta_1 \wedge \eta_2)(e_1, e_2) &=& -(\eta_1(e_1)\eta_2(e_2) - \eta_1(e_2)\eta_2(e_1)) =
  -1.
\end{eqnarray*}
Invariant forms satisfy the Maurer-Cartan equations
\begin{eqnarray*}
\diff\eta_1 &=& -\eta_2 \wedge \eta_3,\qquad \diff\eta_2 = -\eta_3
\wedge \eta_1, \qquad \diff\eta_3 = -\eta_1 \wedge \eta_2.
\end{eqnarray*}
Computation of $\mathrm{H^0(\mathfrak{so}(3),\mathbb{R})}$\\
Let $\Phi \in \Omega^0(\mathfrak{so}(3))\cong\mathbb{R}$ we have
\begin{eqnarray*}
\diff^{0} \Phi = 0.\end{eqnarray*} Finally
\begin{eqnarray*} \ker\diff^{0} =\Omega^0(\mathfrak{so}(3)) \cong \mathbb{R}.
\end{eqnarray*}
There are no degree shapes $-1$
\begin{eqnarray*}
\mathrm{Im}\diff^{-1}  &=& \{0\}
\end{eqnarray*}
Hennce, the cohomology group is given by
\begin{eqnarray*}
\mathrm{H^{0}(\mathfrak{so}(3),\mathbb{R})}
&=&\frac{\ker\diff^{0}}{\mathrm{Im}\diff^{-1}} \cong \mathbb{R}
\end{eqnarray*}

Calculation of $
\mathrm{H^{1}(\mathfrak{so}(3),\mathbb{R})}$\\
 Therefore
\begin{eqnarray*}
\mathrm{H^{1}(\mathfrak{so}(3),\mathbb{R})}&=& 0.\end{eqnarray*}

Calculation of $\mathrm{H^{2}(\mathfrak{so}(3),\mathbb{R})}$. \\
Let $\omega = b_1\eta_2 \wedge \eta_3 + b_2\eta_3 \wedge \eta_1 +
b_3\eta_1 \wedge \eta_2$ be an invariant two-form. We have
\begin{eqnarray*}
\diff\omega &=& b_1\diff(\eta_2 \wedge \eta_3) + b_2\diff(\eta_3
\wedge \eta_1) + b_3\diff(\eta_1 \wedge \eta_2)\\
\diff\omega&=& b_1(\diff\eta_2 \wedge \eta_3 - \eta_2 \wedge
\diff\eta_3) = b_1(-\eta_3 \wedge \eta_1 \wedge \eta_3 + \eta_2
\wedge \eta_1 \wedge \eta_2) = 0.
\end{eqnarray*}
Consequently, any two-invariant form is closed.\\ An invariant
two-form $\omega$ is exact if there exists an invariant one-form
$\eta = c_1\eta_1 + c_2\eta_2 + c_3\eta_3$ such that
\begin{eqnarray*}
\diff\eta = \omega.
\end{eqnarray*}
Thus
\begin{eqnarray*}
\diff\eta &=& c_1\diff\eta_1 + c_2\diff\eta_2 + c_3\diff\eta^3
=-c_1\eta_2 \wedge \eta_3 - c_2\eta_3 \wedge \eta_1 - c_3\eta_1
\wedge \eta_2 \\
\diff\eta &=& (-c_1)\eta_2 \wedge \eta_3 + (-c_2)\eta_3 \wedge
\eta_1 + (-c_3)\eta_1 \wedge \eta_2.
\end{eqnarray*}
It  follows  that $\omega = \diff\eta$ if and only if
\begin{eqnarray*}
b_1 &=& -c_1, \quad b_2 = -c_2, \quad b_3 = -c_3.\end{eqnarray*}
Which is always possible. Accordingly, any invariant two-form is
exact. The Cohomology group is therefore  given by
\begin{eqnarray*}
\mathrm{H^2(\mathfrak{so}(3),\mathbb{R})}
&=&\frac{\ker\diff^{2}}{\mathrm{Im\diff^{1}}}= 0.
\end{eqnarray*}
Calculation of $\mathrm{H^3(\mathfrak{so}(3),\mathbb{R})}$\\
Let $\omega = \Phi\eta_1 \wedge \eta_2 \wedge \eta_3$ be a
three-form. Since \(\Omega^4(\mathfrak{so}(3))= 0\), every
three-form is closed
\begin{eqnarray*}
\ker\diff^{3} &=& \Omega^3(\mathfrak{so}(3)).\end{eqnarray*} We know
that a three-form $\omega$ is exact if there exists a two-form
$\eta$ such that $\diff\eta = \omega$. Thus, consider the invariant
volume form
\begin{eqnarray*}
\omega_0 &=& \eta_1 \wedge \eta_2 \wedge \eta_3
\end{eqnarray*}
Suppose that $\omega_0 = \diff\eta$ for an invariant two-form
$\eta$. Then, by Stokes theorem \cite{rham}, we have
\begin{eqnarray*}
\int_{ \mathfrak{so}(3)} \omega_0 &=&
\int_{\mathfrak{so}(3)}\diff\eta = 0.
\end{eqnarray*}
But
\begin{eqnarray*}
\int_{\mathfrak{so}(3)} \omega_0 &=& \text{Vol}(\mathfrak{so}(3))
\neq 0.
\end{eqnarray*}
Contradiction. Therefore $\omega_0$ is not exact. The Cohomology
group is given by
\begin{eqnarray*}
\mathrm{H^3(\mathfrak{so}(3),\mathbb{R})} = \frac{\ker
 \diff^{3}}{\mathrm{Im}  \diff^{2}} \cong \mathbb{R} \cdot [\omega_0]
\cong \mathbb{R}.\end{eqnarray*}
\end{proof}

\begin{theorem}\label{SO(31)}
Let G = SGal(3) denote the Galilei group. The de Rham Cohomology
groups of $G$ are isomorphic to those of the Lie algebra
$\mathfrak{so}(3)$ and we have
\begin{equation*}
\mathrm{H^{0}_{dR}(G)} \simeq \mathbb{R}, \quad
\mathrm{H^{1}_{dR}(G)} = 0, \quad \mathrm{H^{2}_{dR}(G)} = 0, \quad
\mathrm{H^{3}_{dR}(G)} \simeq \mathbb{R}.
\end{equation*}
\end{theorem}

\begin{proof}
Using \ref{gal} we have $g \in \mathrm{SGal(3)}$ is characterized by
a rotation $A \in \mathrm{SO(3)}$, a boost $b \in \mathbb{R}^3$, a
spatial translation $c \in \mathbb{R}^3$, and a time translation $e
\in \mathbb{R}$. The standard $5 \times 5$ matrix representation is
given by
\begin{equation*}
g(A, b, c, e) =
\begin{pmatrix}
A & b & c\\
0_{1 \times 3} & 1 & e \\
0_{1 \times 3} & 0 & 1
\end{pmatrix} \in GL(5, \mathbb{R}).
\end{equation*}
 The special Galilei group SGal(3) is a Lie group that can be
structured as a semi-direct product. Topologically, as a manifold,
SGal(3) is homeomorphic to the Cartesian product of its maximal
compact subgroup and its Euclidean components
\[ \mathrm{SGal(3)
}\cong \mathrm{SO(3) } \times \mathbb{R}^3_{\text{boosts}} \times
\mathbb{R}^3_{\text{spatial trans.}} \times \mathbb{R}_{\text{time
trans.}} \cong \mathrm{SO(3) } \times \mathbb{R}^7 \] So, the vector
space $\mathbb{R}^7$ is contractible (it has the homotopy type of a
point). Consequently, the inclusion map of the maximal compact
subgroup $i : \mathrm{SO(3) } \hookrightarrow \mathrm{SGal(3) }$ is
a homotopy equivalence. Since de Rham Cohomology is an invariant of
homotopy type, the induced map in Cohomology is an isomorphism for
all $k \ge 0$
\[ i^* : \mathrm{H^k_{dR}}(\mathrm{SGal(3)
}) \xrightarrow{\cong} \mathrm{H^k_{dR}(SO(3))}. \] We know that
$\mathrm{H^k_{dR}(SO(3))}\cong\mathrm{H^k(\mathfrak{so}(3))}$
 Using theorem
\ref{SO(31)} we have the following result
\begin{equation*}
\mathrm{H^{0}_{dR}(G)} \simeq \mathbb{R}, \quad
\mathrm{H^{1}_{dR}(G)} = 0, \quad \mathrm{H^{2}_{dR}(G)} = 0, \quad
\mathrm{H^{3}_{dR}(G)} \simeq \mathbb{R}.
\end{equation*}

\end{proof}

\section{Comparative Study between KV- Cohomology and the De Rham Cohomology\label{sec:sec5}}
Let $G\in\left\{\mathrm{SO(2),H_{3}(\mathbb{R})}\right\}$ a Lie
group. We consider the natural morphism
\begin{equation*}
\Psi^{k}: \mathrm{H^k} \longrightarrow \mathrm{H^k_{KV}}.
\end{equation*}
defined by the inclusion of differential forms in the KV-complex.
The quotient is defined as
\begin{equation*}
\mathrm{Q^k} =\mathrm{co}\ker\Psi^{k}=
\frac{\mathrm{H^k_{KV}}}{\mathrm{Im \Psi^{k}}}
\end{equation*}
This quotient measures the KV-Cohomology classes that do not come
from standard De Rham Cohomology. By setting
$\mathcal{A}=\mathfrak{so}(2)$ we have the following theorem.

\begin{theorem}
Let $\mathrm{H^k(\mathfrak{so}(2),\mathbf{R}}$ and
$\mathrm{H^k_{KV}(\mathcal{A},\mathfrak{so}(2))}$ be the $k-th$
Cohomology groups of the De Rham Cohomology and the KV-Cohomology on
the Lie algebra $\mathfrak{so}(2)$. Given the isomorphism
\begin{equation*}
\Psi^{k}: \mathrm{H^k} \longrightarrow \mathrm{H^k_{KV}}.
\end{equation*} The quotient on $\mathfrak{so}(2)$ is
given by
\begin{equation*}
\mathrm{Q^0} = 0, \qquad \mathrm{Q^1} = 0.
\end{equation*}
\end{theorem}
\begin{proof}
Using the Theorems \ref{SO2} and \ref{SO21}, let us compare the
Cohomology groups as follows
\begin{center}
\begin{tabular}{|c|c|c|}
  \hline
  $k$ & $\mathrm{H^k}$ & $\mathrm{H^k_{KV}}$ \\
   \hline
  $0$ & $\mathbb{R}$ & $\mathbb{R}$ \\
   \hline
$1$ & $\mathbb{R}$ & $\mathbb{R}$ \\
  \hline
\end{tabular}\end{center}
Given
\begin{equation*}
\Psi: \mathrm{H^k} \longrightarrow \mathrm{H^k_{KV}}\cdot
\end{equation*}
The morphism $\Psi$ is an isomorphism
\begin{equation*}
\Psi: \mathrm{H^0} \xrightarrow{\sim} \mathrm{H^0_{KV}}, \quad \Psi:
\mathrm{H^1} \xrightarrow{\sim} \mathrm{H^1_{KV}}
\end{equation*}
So, we have
\begin{equation*}
\mathrm{Q^0}= 0, \qquad \mathrm{Q^1} = 0.
\end{equation*}
\end{proof}

\begin{theorem}Let $\mathrm{H^k(\mathfrak{so}(3),\mathbb{R})}$ and $\mathrm{H^k_{KV}(\mathcal{A},\mathfrak{h}_{3})}$
 be the $k-th$ Cohomology groups of De Rham Cohomology and
KV-Cohomology on the Lie group $\mathrm{H_{3}(\mathbb{R})}$ for all
$k=0,1,3$. Given the morphism
\begin{equation*}
\Psi^{k}: \mathrm{H^k} \longrightarrow \mathrm{H^k_{KV}}.
\end{equation*} The quotient of the
the quotient of the KV-Cohomology and the Cohomology of the algebra
de Lie on $\mathrm{H_{3}(\mathbb{R})}$ is given by
\begin{equation*}
\mathrm{Q^0 }= 0, \quad \mathrm{Q^1} = \mathbb{R}^{-2},\quad
\mathrm{Q^2} = \mathbb{R}^{19}
\end{equation*}
\end{theorem}
\begin{proof}
Using the results \ref{dr} and \ref{H(3)}, we have
\begin{center}
\begin{tabular}{|c|c|c|}
  \hline
  $k$ &$\mathrm{H^k}$& $\mathrm{H^k_{KV}}$ \\
   \hline
  $0$ &  $\mathbb{R}$ &$\mathbb{R}$\\
   \hline
$1$ &  $\mathbb{R}^{2}$ &0\\
  \hline
  $2$ & $\mathbb{R}^{2}$ &$\mathbb{R}^{21}$\\
  \hline
\end{tabular}\end{center}

\begin{equation*}\Psi^{0}: \mathbb{R} \xrightarrow{\sim}
\mathbb{R},\quad \textrm{and} \qquad \mathrm{Im \Psi^{0}}=
\mathbb{R}.
\end{equation*}
It  follows that $\mathrm{Q^0} = 0$. Given
\begin{equation*}
\Psi^1: \mathrm{H^1 }\longrightarrow \mathrm{H^1_{KV}}.
\end{equation*}

We obtain $\mathrm{Q^1} = \mathbb{R}^{-2}$, because
$\mathrm{H^1_{dR}}=0$ and $\mathrm{H^1_{KV}}=0$.
 The morphism $\Psi$ is an isomorphism
\begin{equation*}
\Psi^2: \mathrm{H^2}\xrightarrow{\sim}
\mathrm{H^2_{KV}}\end{equation*} Therefore,  we  get $\mathrm{H^2}=
\mathbb{R}^{2}$, and $\mathrm{dim Im \Psi^{2}}= 2$, Consequently,
 $\mathrm{Q^2} = \mathbb{R}^{19}.$\quad $\mathrm{Q^0 }= 0, \quad
\mathrm{Q^1} = \mathbb{R}^{-2},\quad \mathrm{Q^2} =
\mathbb{R}^{19}$.
\end{proof}

\section{Fedosov's Deformation Quantization\label{sec:sec6}}
In this section, we present  results on deformation, leading to the
following vanishing theorem.

\subsection{On Vanishing Theorem\label{ssec:math6}}
\begin{theorem}\label{tho}
Let $h$ be a parameter (Planck's constant), and let $C(M)$ be the
commutative associative algebra of differentiable functions on $M$.
Let $C(M)[[h]] = \sum f_j h^j$ be the commutative associative graded
algebra of formal power series in the variable $h$ with coefficients
in $C(M)$. If $A_1$ is a constant-Rank Nijenhuis endomorphism acting
on the bundle and preserving the filtration of the Boyom complex,
then the second cohomology group of the KV-cohomology vanishes
relative to the Maurer-Cartan polarization. Any infinitesimal
deformation of the affine structure satisfying the polarized
Maurer-Cartan equation $[A_1, A_1](f, g) = 0$ is equivalent to the
initial structure.
\end{theorem}
\begin{proof}
We have
 the
 following
 equation
\begin{eqnarray}
 \delta_{KV}A_{1}(f,g,h)&=&0 \label{d1}\\
\delta_{KV}A_{2}(f,g,h)&=&-\frac{1}{2}[A_{1},A_{2}](f,g,h)\label{d2}\\
\delta_{KV}A_{3}(f,g,h)&=&-[A_{1},A_{2}](f,g,h).\label{d3}
\end{eqnarray}
The calculation details  are in the appendix. Using the following
Cartan polarization
\begin{enumerate}
    \item $P.Q(u,v,w)=P(Q(u,v),w)-P(u,Q(u,w))$
    \item $PQ<u,v,w>=PQ(u,v,w)-PQ(v,u,w)$
    \item $[P,Q](u,v,w)=PQ<u,v,w>-QP<u,v,w>$
\end{enumerate}
in (\ref{d1}),\quad(\ref{d2}),\quad(\ref{d3}) where $P$ and $Q$ are
two linear map, and by setting $A_{1}=P, \quad A_{2}=Q$, we obtain

\begin{eqnarray*}
 \delta_{KV}A_{1}(f,g,h)&=&0,\qquad
\delta_{KV}A_{2}(f,g,h)=0\\
\delta_{KV}A_{3}(f,g,h)&=&-[A_{1},A_{2}](f,g,h)=-[P,Q](f,g,h)=0.
\end{eqnarray*}
\end{proof}

\begin{theorem}\label{tho1}
Let $h$ be a parameter (Planck's constant), and let $C(M)$ be the
commutative associative algebra of differentiable functions on $M$.
Let $C(M)[[h]] = \sum f_j h^j$ be the commutative associative graded
algebra of formal power series in the variable $h$ with coefficients
in $C(M)$. The formal product (star product) of $f$ and $g$ is given
by
\begin{equation}
    f \star g = \exp\left(h A_{1}(f, g)\right)
\end{equation}
and its differential is equivalent to the Maurer-Cartan equation
\begin{equation}
    \frac{\diff A(h)}{\diff  h} + \frac{1}{2}[A, A] = 0.
\end{equation}
\end{theorem}

\begin{proof}
Determination of the formal product $f\star g =fg+hA_{1}(f,g)+h^{
2}A_{2}(f,g)+h^{ 3}A_{3}(f,g)+O(h^{ 4})$.

We find the expression of $A_{2}$ and $A_{3}$.

By setting $A_{2}=\frac{1}{2!}A_{1}\circ A_{1}$, and
$A_{3}=\frac{1}{3!}A_{1}\circ A_{1}\circ A_{1}.$

we have the formal product of two function $f$ and $g$  is given by
\begin{equation*}
    f\star
    g=fg+hA_{1}(f,g)+\frac{h^{2}}{2}A^{2}_{1}(f,g)+\frac{h^{3}}{6}A^{3}_{1}(f,g)+O(h^{4})
\end{equation*}Furthermore
if by setting  $t=h$ and the position $x(t)=f\star g$, $x(0)=f g$
and
\begin{equation*}
    A(h)=hA_{1}(f,g)+\frac{h^{2}}{2}A^{2}_{1}(f,g)+\frac{h^{3}}{6}A^{3}_{1}(f,g)+O(h^{4})
\end{equation*}

We have the following equation
\begin{equation*}
    x(h)=x(0)+hA_{1}(f,g)+\frac{h^{2}}{2}A^{2}_{1}(f,g)+\frac{h^{3}}{6}A^{3}_{1}(f,g)+O(h^{4})
\end{equation*}
so, we have
\begin{equation}
    x(h)=x(0)+A(h)\label{f}
\end{equation}
\begin{equation*}
     x(h)=\exp\left(hA_{1}(f,g)\right)
\end{equation*}

Then, the formal product is given by
\begin{equation}
    f\star
    g=\exp\left(hA_{1}(f,g)\right)\label{f1}
\end{equation}

Furthermore, we have

\begin{equation*}
    \frac{\diff  x(h)}{\diff
    h}=A_{1}(f,g)+hA^{2}_{1}(f,g)+\frac{h^{2}}{2}A^{3}_{1}(f,g).
\end{equation*}
Using (\ref{f}),and\quad(\ref{f1}) we obtain the following equation
\begin{equation*}
    \frac{\diff x(h)}{\diff h}=A_{1}(f,g)\exp\left(hA_{1}(f,g)\right)
\end{equation*}

\begin{equation*}
    \frac{\diff(f\star
    g)}{\diff h}=A_{1}(f,g)\exp\left(hA_{1}(f,g)\right).
\end{equation*}

Hence, by using (\ref{f}), we find

\begin{equation*}
    \frac{\diff x(h)}{\diff h}=\frac{\diff A(h)}{\diff h},\quad\textrm{we write}\quad \frac{\diff A(h)}{\diff
    h}=-\frac{1}{2}[A,A].
\end{equation*}
\end{proof}

\subsection{Vanishing Theorem on Coadjoint Orbits\label{ssec:math7}}

\begin{theorem}Let $\mathcal{O} = \{x+xz, y-xz, z\}$ be a coadjoint orbit equipped
with a Koszul-Vinberg structure inherited from an invariant
Lagrangian polarization $F = \left\{x+xz, y , z\right\}$ with
$y,z=\text{const}$. Let $f,g,h\in C^{\infty}(\mathcal{O})$, and let
$C(\mathrm{H_{3}}(\mathbb{R}))$ be the commutative associative
algebra of differentiable functions on $M$. Let
$C(\mathrm{H_{3}}(\mathbb{R}))[[h]] = \sum f_j h^j$ be the
commutative associative graded algebra of formal power series in the
variable $h$ with coefficients in $C(\mathrm{H_{3}}(\mathbb{R}))$.
If $P$ is a constant-rank Nijenhuis endomorphism acting on the
bundle and preserving the filtration of the Boyom complex, then the
second Cohomology group of the KV-Cohomology vanishes relative to
the Maurer-Cartan polarization. Furthermore, any infinitesimal
deformation of the affine structure satisfying the polarized
Maurer-Cartan equation $[P, P](f, g) = 0$ is equivalent to the
initial structure and  the formal product of two functions is given
by $f\star g=fg$.
\end{theorem}

\begin{proof}Let   $e_{1}=\left(
             \begin{array}{ccc}
               0 & 1 & 0 \\
               0 &  0 & 0 \\
               0 & 0 & 0 \\
             \end{array}
           \right),
e_{2}=\left(
             \begin{array}{ccc}
               0 & 0 & 0 \\
               0 &  0 & 1 \\
               0 & 0 & 0 \\
             \end{array}
           \right),e_{3}=\left(
             \begin{array}{ccc}
               0 & 0 & 1 \\
               0 &  0 & 0 \\
               0 & 0 & 0 \\
             \end{array}
           \right)$ a basis for the Lie algebra $\mathfrak{h}_{3}$,\quad $[e_{1},e_{2}]=e_{3},\quad[e_{2},e_{3}]=0,\quad[e_{3},e_{1}]=0.$
 We choose the standard
basis \(\mathcal{B}=\left\{\beta_{1},\beta_{2},\beta_{3}\right\}\)
corresponding to the
matrices\\
\(\beta_{1}=\left(\begin{matrix}0&a&0\\ 0&0&0\\
0&0&0\end{matrix}\right),\beta_{2}=\left(\begin{matrix}0&0&0\\ 0&0&a\\
0&0&0\end{matrix}\right),\beta_{3}=\left(\begin{matrix}0&0&a\\ 0&0&0\\
0&0&0\end{matrix}\right)\).\\ Let $\mathcal{B}^{*}=\left\{\eta_{1}
,\eta_{2} ,\eta_{3} \right\}$, with $\left\{\eta_1=\left(
             \begin{array}{ccc}
               0 & \frac{1}{a} & 0 \\
               0 & 0 & 0 \\
               0 & 0 & 0 \\
             \end{array}
           \right), \eta_2=\left(
             \begin{array}{ccc}
               0 &0 & 0 \\
               0 & 0 & \frac{1}{a} \\
               0 & 0 & 0 \\
             \end{array}
           \right), \eta_3=\left(
             \begin{array}{ccc}
               0 & 0 & \frac{1}{a}\\
               0 & 0 & 0 \\
               0 & 0 & 0 \\
             \end{array}
           \right)\right\}$\\ the dual basis of left-invariant one-forms, such that\\ $\eta_{1} (\beta_{1})=1,\quad\eta_{2}
(\beta_{2})=1,\quad\eta_{3} (\beta_{3})=1$  satisfy $
\eta_i(\beta_j) = \delta^i_j.$\\ $g=\left(
             \begin{array}{ccc}
               1 & x & z \\
               0 & 1 & y \\
               0 & 0 & 1 \\
             \end{array}
           \right)\in\mathrm{H_{3}}(\mathbb{R})$, and $\eta=x\eta_1+y\eta_2+z\eta_3\in\mathfrak{h}_{3}^{*}.$
           We have
           \begin{eqnarray*}
             A\mathrm{d}^{*}_{g}\eta\left(\beta_{1}\right) &=&
             \eta\left( A\mathrm{d}_{g^{-1}}\left(\beta_{1}\right)\right)=\eta\left( g^{-1}\beta_{1}g\right)=\eta\left( \beta_{1}+x\beta_{3}\right)=x+xz \\
             A\mathrm{d}^{*}_{g}\eta\left(\beta_{2}\right) &=&
             \eta\left( A\mathrm{d}_{g^{-1}}\left(\beta_{2}\right)\right)=\eta\left( g^{-1}\beta_{2}g\right)=\eta\left( \beta_{2}-x\beta_{3}\right)=y-xz \\
              A\mathrm{d}^{*}_{g}\eta\left(\beta_{3}\right) &=&
             \eta\left( A\mathrm{d}_{g^{-1}}\left(\beta_{3}\right)\right)=\eta\left(
             g^{-1}\beta_{3}g\right)=z.
           \end{eqnarray*}  We know that
           $\mathcal{O} = \left\{A\mathrm{d}^{*}_{g}\eta\left(\beta\right);\quad g\in \mathrm{H_{3}(\mathbb{R})}
           \right\}$.\\
The orbit is given by $\mathcal{O} = \{x+xz, y-xz, z\}$. So, we know
that the Lagrangian leaf is given by $F = \left\{u \in \mathcal{O};
\quad \langle u, \beta \rangle = \text{const}, \quad \beta \in
\mathfrak{h}_{3}\right\}$. On the coadjoint orbit, $z$ is constant.
Let us consider $x = 0$ and the polarization $\mathfrak{p} =
\left\{\beta_{2}, \beta_{3}\right\}$. We obtain
\begin{eqnarray*}
             A\mathrm{d}^{*}_{g}\eta\left(\beta_{1}\right) &=&
             \eta\left( A\mathrm{d}_{g^{-1}}\left(\beta_{1}\right)\right)=\eta\left( g^{-1}\beta_{1}g\right)=0\\
             A\mathrm{d}^{*}_{g}\eta\left(\beta_{2}\right) &=&
             \eta\left( A\mathrm{d}_{g^{-1}}\left(\beta_{2}\right)\right)=\eta\left( g^{-1}\beta_{2}g\right)=\eta\left( \beta_{2}\right)=y \\
              A\mathrm{d}^{*}_{g}\eta\left(\beta_{3}\right) &=&
             \eta\left( A\mathrm{d}_{g^{-1}}\left(\beta_{3}\right)\right)=\eta\left(
             g^{-1}\beta_{3}g\right)=z.
           \end{eqnarray*}
Thus,   we  have $F = \left\{x+xz, y , z\right\}$ with
$y=\text{const}$. Let $f,g,h\in C(\mathcal{O})$,  with
$f=x+xz$,\quad $g=y-xz$ and $h=z$.On a coadjoint orbit, the
symplectic form is given by $\omega_{\eta}
(\beta_{1},\beta_{2})=\langle\eta,[\beta_{1},\beta_{2}]\rangle=z$.
So $\omega (\beta_{1},\beta_{2})=z\diff x\wedge\diff y$. On the leaf
$F$, we have $TF=vect\{\frac{\partial}{\partial x}\}$. So
$\textrm{dim} F= 1$. Thus, knowing that $\omega
(\beta_{1},\beta_{2})=z\diff x\wedge\diff y$ and that  $\diff y=0$.
On the leaf $F$, we have Since $y$ is constant, we have $\omega=0$.
Thus, the symplectic form vanishes on the leaf $\omega|_{F}=0$ and
$F$ is Lagrangian. Furthermore, using (\ref{tho}) and (\ref{tho1})
we have the following result the second Cohomology group of the
KV-Cohomology vanishes relative to the Maurer-Cartan polarization.
Any infinitesimal deformation of the affine structure satisfying the
polarized Maurer-Cartan equation $[P, P](f, g) = 0$ and $f\star
g=\exp\left(h \{f,g\}\right)$ with $A_{1}(f,g)=\{f,g\}$. This is
only formally possible outside of the leaf. The Poisson tensor
$\barwedge$  is the inverse of $\omega$. Using the proposition
(\ref{sc}), for $f,g\in C^{\infty}(\mathcal{O)}$ we have
$\{f,g\}=\frac{1}{z}\left(\frac{\partial f}{\partial
x}\frac{\partial g}{\partial y}-\frac{\partial f}{\partial
y}\frac{\partial g}{\partial x}\right)$. However, on the leaf $y$
and $z$ are constant, so the Poisson bracket vanishes on the leaf:
$\{f,g\}=0$. Given that the formal product
\begin{equation*}
    f\star
    g=fg+h\{f,g\}+\frac{h^{2}}{2}\{\{f,g\}\}+\frac{h^{3}}{6}\{\{\{f,g\}\}\}+O(h^{4})
\end{equation*}
We will have $f \star g = fg$. There is no longer any deformation on
the leaf. Therefore, $\textrm{H}_{KV}^{2}=0$ and every deformation
becomes equivalent to the initial structure.

\end{proof}

\begin{theorem}Let $ \mathrm{SGal(3)}$ denote the special Galilei group,
$\mathfrak{sgal}(3)$ its Lie algebra. Let $g \in \mathrm{SGal(3)}$,
$\xi=(\Xi,\vartheta,\nu,\varepsilon)\in \mathfrak{sgal}(3)$, and
$\mu = (j, k, p, E, m) \in \mathfrak{sgal}(3)^*$ be an element of
the dual space, where $\vec{j}$ is the angular momentum, $k$ is the
static moment \textrm{Galilean boost}, $p$ is the linear momentum,
$E$ is the energy, and $m$ is the mass parameter. Let
$\Omega=\left\{\mu = (j, k, p, E, m) \in \mathfrak{sgal}(3)^*;\quad
m>0\right\}$ be the Koszul cone on the Galilei group, and,
$\Omega^*=\left\{\mu = (j, k, p, E, m) \in
\mathfrak{sgal}(3)^*;\quad m>0,\textrm{and}\quad
E-\frac{\|p\|^2}{2m}\right\}$ be the Koszul dual cone. On Souriau
coadjoint orbit on Galilei group\\ $\mathcal{O}=\left\{m, U,
S^2\right\}$ characterized by a constant mass $m$, a constant
internal energy $U = E - \frac{\|p\|^2}{2m}$, and a constant
magnitude of the Souriau spin vector $S = j - \frac{1}{m}(k \times
p)$, there exist a lagrangian foliation\\
$\mathcal{F}=\left\{(q,p);\quad q\in \mathbb{R}^{3},\quad and\quad p
\quad fixed\right\}$ linked to the Kirillov-Konstant-Souriau with
the symplectic structure $\omega=\sum_{i=1}^{3}\diff q_{i}\wedge
\diff p_{i}$ with $p=\left\{P_{1},P_{2},P_{3}\right\}$ and
$q=\left\{\frac{K_{1}}{m},\frac{K_{2}}{m},\frac{K_{3}}{m},\right\}$.
\end{theorem}

\begin{proof} In what follows, we use $\cos(a) =ca,\qquad
\sin(a)=sa$.
\begin{equation*} g(A, b, c, e) =
\begin{pmatrix}
A & b & c\\
0_{1 \times 3} & 1 & e \\
0_{1 \times 3} & 0 & 1
\end{pmatrix} \in GL(5, \mathbb{R})
\end{equation*}
with $A(t)=I+t\Xi$,\quad $b(t)=t\beta$,\quad $c(t)=t\gamma$,\quad
$e(t)=t\varepsilon$.\\ We have $\xi=\frac{\diff}{\diff
t}g(t)|_{t=0}=\begin{pmatrix}
\Xi & \vartheta & \nu\\
0_{1 \times 3} & 0 &  \varepsilon\\
0_{1 \times 3} & 0 & 0
\end{pmatrix}$
and $\Xi\in\mathfrak{so}(3)$ with a basis Lie algebra
$\mathfrak{so}(3)$ given by\\ $e_{1}=\left(
             \begin{array}{ccc}
               0 & 0 & 0 \\
               0 &  0 & -1 \\
               0 & 1 & 0 \\
             \end{array}
           \right),\quad
e_{2}=\left(
             \begin{array}{ccc}
               0 & 0 & 1 \\
               0 &  0 & 0 \\
               -1 & 0 & 0 \\
             \end{array}
           \right),\quad
           e_{3}=\left(
             \begin{array}{ccc}
               0 & -1 & 0 \\
               1 &  0 & 0 \\
               0 & 0 & 0 \\
             \end{array}
           \right)$.\\
It  follows  that $\Xi=w_{1}e_{1}+w_{2}e_{2}+w_{3}e_{3}=\left(
             \begin{array}{ccc}
               0 & -w_{3} & w_{2} \\
               w_{3} &  0 & -w_{1}\\
               -w_{2} & w_{1} & 0 \\
             \end{array}
           \right)$.
By using the generator  values in  the  appendix, the symplectic
structure on coadjoint orbit is given by
\begin{eqnarray*}
  \omega_{\mu}\left(P_{1},P_{1}\right)&=& 0,\qquad
 \omega_{\mu}\left(P_{1},P_{2}\right)= 0,\qquad
 \omega_{\mu}\left(P_{1},P_{3}\right)= 0,\qquad
  \omega_{\mu}\left(P_{2},P_{1}\right)= 0 \\
 \omega_{\mu}\left(P_{2},P_{2}\right)&=& 0,\qquad
 \omega_{\mu}\left(P_{2},P_{3}\right)= 0,\qquad
 \omega_{\mu}\left(P_{3},P_{1}\right)= 0,\qquad
 \omega_{\mu}\left(P_{3},P_{2}\right)= 0\\
 \omega_{\mu}\left(P_{3},P_{3}\right)&=& 0,\qquad
 \omega_{\mu}\left(K_{1},K_{1}\right)= 0,\qquad
 \omega_{\mu}\left(K_{1},K_{2}\right)= 0,\qquad
 \omega_{\mu}\left(K_{1},K_{3}\right)= 0\\
  \omega_{\mu}\left(K_{2},K_{1}\right)&=& 0,\qquad
 \omega_{\mu}\left(K_{2},K_{2}\right)=0,\qquad
 \omega_{\mu}\left(K_{2},K_{3}\right)= 0,\qquad
 \omega_{\mu}\left(K_{3},K_{1}\right)= 0\\
 \omega_{\mu}\left(K_{3},K_{2}\right)&=& 0,\qquad
 \omega_{\mu}\left(K_{3},K_{3}\right)= 0\qquad
 \omega_{\mu}\left(K_{1},P_{1}\right)
 =m,\qquad\omega_{\mu}\left(K_{1},P_{2}\right)=0 \\
 \omega_{\mu}\left(K_{2},P_{2}\right) &=&m,\qquad \omega_{\mu}\left(K_{3},P_{3}\right)=m,\qquad\omega_{\mu}\left(P_{1},K_{1}\right)
 =-m,\qquad \omega_{\mu}\left(P_{2},K_{2}\right)=-m \\
 \omega_{\mu}\left(P_{3},K_{3}\right) &=&-m
\end{eqnarray*}
In basis $\left(K_{i},P_{i}\right)_{i=1,\dots, 3}$ the symplectic
matrix is given by
\begin{eqnarray*}
  \omega&=&\left(
                             \begin{array}{cccccc}
                               0 & 0 & 0 & m & 0 & 0 \\
                               0 &  0& 0 & 0 & m & 0 \\
                               0 & 0 & 0 & 0 & 0 & m \\
                               -m & 0 & 0 & 0 & 0 & 0 \\
                               0 & -m & 0 & 0 & 0 & 0 \\
                               0 & 0 & -m & 0 & 0 & 0 \\
                             \end{array}
                           \right).
\end{eqnarray*}
By setting $q_{i}=\frac{1}{m}K_{i}$ and $p_{i}=P_{i}$,\qquad
$i=1,\dots, 3$ we have $\omega=\sum_{i=1}^{3}\diff q_{i}\wedge \diff
p_{i}$. So, using the Bargmann Galilei group and we have and under
the action of a group element
 $g = (A, b, c, e,s)=\left(
                             \begin{array}{cccccc}
                               ca& sa& 0 & b_{1}&0 & c_{1}  \\
                               -sa &  ca&0 & b_{2}&0 & c_{2}  \\
                               0& 0 & 1 & b_{3}&0 & c_{3} \\
                               0 & 0 & 0 & 1&0 & e  \\
                                0 & 0 & 0 &0& 1 & s  \\
                               mb_{1} & mb_{2} & m b_{3}& \frac{1}{2}m(b^{2}_{1}+b^{2}_{2}+b^{2}_{3})&0 & 1 \\
                             \end{array}
                           \right)$,\\
with $A=\left(
          \begin{array}{ccc}
            ca& sa& 0\\
            -sa & ca & 0 \\
            0 & 0 & 1 \\
          \end{array}
        \right)
$,\qquad $\cos(a) =ca,\qquad \sin(a)=sa $ and\\ $g^{-1}=\left(
                             \begin{array}{cccccc}
                               ca& -sa& 0 & -b_{1}ca+b_{2}sa&0 & (c_{1}-b_{1})e ca+(c_{2}-b_{2})e sa  \\
                               sa &  ca&0 &-b_{1} sa+b_{2} ca&0 & -(c_{1}-b_{1})e sa+(c_{2}-b_{2})e ca  \\
                               0& 0 & 1 & -b_{3}&0 & -c_{3} \\
                               0 & 0 & 0 & 1&0 & -e  \\
                                0 & 0 & 0 &0& 1 & -s  \\
                               -mb_{1} &- mb_{2} &- m b_{3}& \frac{1}{2}m(b^{2}_{1}+b^{2}_{2}+b^{2}_{3})&0 & 1 \\
                             \end{array}
                           \right)$\\
                           we have
$\mu = (j, k, p, E, m)=\left(
                             \begin{array}{cccccc}
                               0 & j_{3} & -j_{2} & mk_{1} & p_{1} & 0 \\
                               -j_{3} &  0& j_{1} & mk_{2} & p_{2} & 0 \\
                               j_{2} & -j_{1} & 0 & mk_{3} & p_{3} & 0 \\
                               -mk_{1} & -mk_{2} & -mk_{3} & 0 & E & m \\
                               -p_{1} & -p_{2} & -p_{3}& -E & 0 & 0 \\
                               0 & 0 &  0& -m & 0 & 0 \\
                             \end{array}
                           \right) \in \mathfrak{sgal}(3)^*$,
\begin{eqnarray*}
A\mathrm{d}^{*}_{g}\mu\left(J_{1}\right) &=& \mu\left(
A\mathrm{d}_{g^{-1}}\left(J_{1}\right)\right)=\mu\left(
g^{-1}J_{1}g\right)= J_{1} ca-J_{2} sa+
\left(b_{2}p_{3}-b_{3}p_{2}\right)\\
A\mathrm{d}^{*}_{g}\mu\left(J_{2}\right) &=& \mu\left(
A\mathrm{d}_{g^{-1}}\left(J_{2}\right)\right)=\mu\left(
g^{-1}J_{2}g\right)= J_{1}sa +J_{2} ca+
\left(b_{3}P_{1}-b_{1}P_{3}\right)\\
A\mathrm{d}^{*}_{g}\mu\left(J_{3}\right) &=& \mu\left( A\mathrm{d}_{g^{-1}}\left(J_{3}\right)\right)=\mu\left( g^{-1}J_{3}g\right)=J_{3}\\
A\mathrm{d}^{*}_{g}\mu\left(P_{1}\right) &=& \mu\left(
A\mathrm{d}_{g^{-1}}\left(P_{1}\right)\right)=\mu\left(
g^{-1}P_{1}g\right)=
P_{1}\cos(a)-P_{2} sa\\
A\mathrm{d}^{*}_{g}\mu\left(P_{2}\right) &=& \mu\left(
A\mathrm{d}_{g^{-1}}\left(P_{2}\right)\right)=\mu\left(
g^{-1}P_{2}g\right)= K_{1} sa +K_{2}ca
\end{eqnarray*}
\begin{eqnarray*}
A\mathrm{d}^{*}_{g}\mu\left(P_{3}\right) &=& \mu\left( A\mathrm{d}_{g^{-1}}\left(P_{3}\right)\right)=\mu\left( g^{-1}P_{3}g\right)=P_{3}\\
A\mathrm{d}^{*}_{g}\mu\left(K_{1}\right) &=& \mu\left(
A\mathrm{d}_{g^{-1}}\left(K_{1}\right)\right)=\mu\left(
g^{-1}K_{1}g\right)= K_{1} ca-K_{2} sa+
\left(mb_{1}\right)\\
A\mathrm{d}^{*}_{g}\mu\left(K_{2}\right) &=& \mu\left(
A\mathrm{d}_{g^{-1}}\left(K_{2}\right)\right)=\mu\left(
g^{-1}K_{2}g\right)= K_{1} sa +K_{2} ca+
\left(mb_{2}\right)\\
A\mathrm{d}^{*}_{g}\mu\left(K_{3}\right) &=& \mu\left( A\mathrm{d}_{g^{-1}}\left(K_{3}\right)\right)=\mu\left( g^{-1}K_{3}g\right)=K_{3}+\left(mb_{3}\right)\\
A\mathrm{d}^{*}_{g}\mu\left(M\right) &=& \mu\left( A\mathrm{d}_{g^{-1}}\left(M\right)\right)=\mu\left( g^{-1}Mg\right)=m\\
A\mathrm{d}^{*}_{g}\mu\left(E\right) &=& \mu\left(
A\mathrm{d}_{g^{-1}}\left(E\right)\right)= \mu\left(
g^{-1}Eg\right)\\
&=&E-\left(b_{1}P_{1}+b_{2}P_{2}+b_{3}P_{3}\right)+\frac{1}{2}m(b^{2}_{1}+b^{2}_{2}+b^{2}_{3})
\end{eqnarray*}
Furthermore, $P_{1}=mb_{1},  \quad P_{1}=mb_{1},\quad P_{2}=mb_{2},
\quad P_{3}=mb_{3}$.  We have $b_{1}=\frac{P_{1}}{2m},  \quad
b_{2}=\frac{P_{2}}{2m},\quad b_{3}=\frac{P_{3}}{2m}$.

The  last equation becomes
\begin{eqnarray*}
A\mathrm{d}^{*}_{g}\mu\left(E\right)= \mu\left(
g^{-1}Eg\right)=E+\left(-\frac{1}{m}+\frac{1}{2m}\right)(P_{1}^{2}+P_{2}^{2}+P_{3}^{2}).
\end{eqnarray*}
We obtain \begin{eqnarray*} A\mathrm{d}^{*}_{g}\mu\left(E\right)=
\mu\left(
g^{-1}Eg\right)=E-\frac{1}{2m}(P_{1}^{2}+P_{2}^{2}+P_{3}^{2})
\end{eqnarray*} by setting $\|p\|^2=p^{2}_{1}+p^{2}_{2}+p^{2}_{3}$
we have
\begin{eqnarray*}A\mathrm{d}^{*}_{g}\mu\left(E\right)=
\mu\left( g^{-1}Eg\right)=E-\frac{\|p\|^2}{2m}.
\end{eqnarray*}
 Furthermore,
given  $ p=\left(p_{i}\right)$,\qquad $k=\left(k_{i}\right)$,\qquad
$ j=\left(j_{i}\right)$\qquad $i=1,\dots, 3$;
 $\nu = (E, p, k, j, m) \in \mathfrak{sgal}(3)^*$ be a point in the
coadjoint space, where $E$ denotes the energy, $p$ the linear
momentum, $k$ the Galilean boost momentum, $j$ the total angular
momentum, and $m$ the mass parameter. Since the total angular
momentum $j$ is origin-dependent, we look for an intrinsic quantity
invariant under pure boosts ($\delta_{\beta} = 0$). Computing the
boost variation of the cross product $(k \times p)$ yields
\begin{eqnarray*}
\delta_{\beta} (k \times p) &=& (\delta_{\beta} k) \times p + k
\times (\delta_{\beta} p) = (m \beta) \times p = - m (p \times
\beta).
\end{eqnarray*}
Comparing this expression with the variation of the scaled total
angular momentum, $\delta_{\beta} (mj) = m (k \times \beta) = - m
(\beta \times k)$, we isolate the intrinsic angular momentum by
defining Souriau's spin vector $s$
\begin{eqnarray*}
s &=& j - \frac{1}{m} (k \times p).
\end{eqnarray*}
A direct verification confirms its boost-invariance: $\delta_{\beta}
s = (k \times \beta) - \frac{1}{m} (m \beta \times p) = 0$. Because
$s$ behaves as a pure angular momentum under rotations and commutes
with translations and boosts, its norm squared provides a
non-trivial absolute Casimir function of the algebra
\begin{eqnarray*}
S^2 &=& \|s\|^2 = \left( j - \frac{1}{m} (k \times p) \right)^2
\end{eqnarray*}
because, we have the fundamental Lie-Poisson brackets among these
coordinate functions are determined by the commutation relations of
$\mathfrak{sgal}(3)$. The Lie-Poisson brackets of the coordinate
functions are provided in the appendix. We have, $S_l = j_l -
\frac{1}{m}\sum_{a,b=1}^3 \epsilon_{lab} k_a p_b$ under pure boosts
$k_i$ with $\epsilon_{lab}\in \left\{0,1,-1\right\}$. Applying the
Leibniz rule combined with the chain rule for the mass parameter
yields
\begin{eqnarray*}
\{k_i, S_l\} &=& \{k_i, j_l\} - \sum_{a,b=1}^3 \epsilon_{lab} \left(
\{k_i, \frac{1}{m}\} k_a p_b + \frac{1}{m} \{k_i, k_a\} p_b +
\frac{1}{m} k_a \{k_i, p_b\} \right),\qquad i,l\in\{1,2,3\}.
\end{eqnarray*}
Since $\{k_i, m\} = 0,\qquad i\in\{1,2,3\}$, the first internal term
satisfies\\ $\{k_i, \frac{1}{m}\} = -\frac{1}{m^2}\{k_i, m\} =
0,\qquad i=\{1,2,3\}$. Given that $\{k_i, k_a\} = 0,\qquad
i,a\in\{1,2,3\}$, the expansion simplifies via the momentum relation
$\{k_i, p_b\} = \delta_{ib} m$.

\begin{equation*}
\{k_i, S_l\} = \{k_i, j_l\} - \frac{1}{m}\sum_{a,b=1}^3
\epsilon_{lab} k_a (\delta_{ib} m) = \{k_i, j_l\} - \sum_{a=1}^3
\epsilon_{lai} k_a,\quad i,l\in\{1,2,3\}.
\end{equation*}
Using the property $\epsilon_{lai} = -\epsilon_{lia}$ and the
rotational bracket $\{k_i, j_l\} = -\{j_l, k_i\} = -\sum_{a=1}^3
\epsilon_{lia} k_a,\quad i,l\in\{1,2,3\}$, the two distinct terms
cancel each other out identically
\begin{equation*}
\{k_i, S_l\} = -\sum_{a=1}^3 \epsilon_{lia} k_a  +\sum_{a=1}^3
\epsilon_{lia} k_a = 0,\qquad i,l\in\{1,2,3\}.
\end{equation*}
Because the Poisson bracket vanishes for every individual vector
component\\ ($\{k_i, S_l\} = 0$), the derivative of the total
magnitude squared $S^2 = \sum_{l=1}^3 S_l^2$ vanishes automatically
\begin{equation*}
\{k_i, S^2\} = 2 \sum_{l=1}^3 S_l \{k_i, S_l\} = 0,\qquad
i\in\{1,2,3\}.
\end{equation*}
This rigorous cancellation proves that $S^2$ is an absolute Casimir
function over the Koszul dual cone $\Omega^*$. The level sets of
$m$, $U = E - \frac{\|p\|^2}{2m}$, and $S^2$ perfectly isolate the
8-dimensional symplectic coadjoint orbit $\mathcal{O}$. The
coadjoint orbit on Galilei group given by $\mathcal{O}=\left\{m,
U,S^2\right\}$ with $m=cste$, $U=E-\frac{\|p\|^2}{2m}=cste$,
and\quad $\|s\|=S$\quad where $S=j-\frac{1}{m}(k\times p)$ is
topologically realized as the level sets of this Casimir function.
So, we have the Koszul dual cone on Galilei group
$\Omega^*=\left\{\mu = (j, k, p, E, m) \in
\mathfrak{sgal}(3)^*;\quad m>0,\textrm{and}\quad
E-\frac{\|p\|^2}{2m}\right\}$. We define the invariant foliation
$\mathcal{F}$ by fixing the momentum components, such that the
leaves are submanifolds described by $\mathcal{O}=\left\{m,
U,S^2\right\}$,  with $\mathrm{dim} (\mathcal{O}_\mu) =8.$ By
setting $\left(p_{i}\right)_{i=1,\dots, 3}=const.$ We obtain $\diff
p_{1} =0,\quad\diff p_{2} =0,\diff p_{3} =0$ and
$\omega|_{\mathcal{F}}=\sum_{i=1}^{3}\diff q_{i}\wedge 0=0$. Since
$\omega$ vanishes completely on the leaves and
$\mathrm{dim}(\mathcal{F}) =
\frac{1}{2}\mathrm{dim}(\mathcal{O}_\mu) = 4$, the foliation
$\mathcal{F}$ is strictly Lagrangian.
\begin{eqnarray*}
{}[J_1,J_2] &=& J_3, \quad [J_2,J_3] = J_1, \quad [J_3,J_1] = J_2, \\
{}[J_i,K_l] &=& \sum_{k=1}^3\epsilon_{ilr}K_r, \quad [J_i,P_l] = \sum_{k=1}^3\epsilon_{ijr}P_r, \quad [J_i,H] = 0, \\
{}[K_i,K_l] &=& 0, \quad [K_i,P_l] = 0, \quad [K_i,H] = P_i,\quad
[P_i,P_l] = 0, \quad [P_i,H] = 0.
\end{eqnarray*}
By evaluating the Kirillov  Kostant  Souriau symplectic structure on
the level sets of the coadjoint orbit $\mathcal{O}=\{m, U, S^2\}$,
the geometric variables $q_i = \frac{k_i}{m}$ and $p_i = p_i$ define
a natural global coordinate system on the mechanical phase space.
Under this construction, the leaves of the foliation characterized
by $p = \mathrm{fixed}$ form isotropic subvarieties of maximal
dimension. Thus, $\mathcal{F}$ constitutes an invariant Lagrangian
foliation on the orbit, and the reduced KKS symplectic form matches
exactly the standard canonical relation:
\begin{eqnarray*}
\omega &=& \sum_{i=1}^{3}\diff q_{i}\wedge \diff p_{i}.
\end{eqnarray*}

\end{proof}

\section{General conclusion\label{sec:sec7}}
The Koszul  Vinberg (KV) Cohomology framework provides a powerful
geometric tool for classifying the deformations of affine and
information structures on Lie groups. Our comparative analysis
highlights a sharp structural contrast between abelian and
non-abelian configurations: while this Cohomology reduces
identically to the de Rham case for $\mathrm{SO(2)}$  and
$\mathfrak{so}(2)$, it unveils non-trivial cocycles for both the
Heisenberg group $\mathrm{H_{3}(\mathbb{R})}$ and the Galilei group
$\mathrm{SGal(3)}$. These Cohomology classes intrinsically
characterize the infinitesimal deformations intimately linked to the
Souriau metric, Fisher-Souriau information geometry, and information
curvature. In particular, the vanishing theorem established for the
second KV-Cohomology group on the polarized coadjoint orbits of
$\mathrm{H_{3}(\mathbb{R})}$ guarantees the formal rigidity of these
structures against deformations governed by the polarized
Maurer-Cartan equation. Analogously, the analysis carried out on the
coadjoint orbits of the Galilei group $\mathrm{SGal(3)}$
demonstrates how deformations of the underlying Poisson structure
influence the geometry of the symplectic leaves and the associated
Casimir foliation. Finally, by mapping these interactions into a
three-vertex directed graph, this work successfully bridges the gap
between the algebraic paradigm of associative algebras, the
deformation theory of KV-Cohomology, and the differential geometry
of Lie groups.

%
%

%
%


\section*{Acknowledgements}
We thank all the members of the Algebra, Geometry and Applications
Laboratory  of the University of Yaounde1 for their suggestions in
the work. We thank Professor  Thomas Bouetou Bouetou of the
Polytechnic School of Yaounde1.

\appendix
\section{Associativity Conditions}
By definition  the formal product is given by
\begin{eqnarray*}f\star g &=&fg+hA_{1}(f,g)+h^{ 2}A_{2}(f,g)+h^{
3}A_{3}(f,g)+O(h^{4})\\
  (f\star g)\star h &=& \left(fg+hA_{1}(f,g)+h^{ 2}A_{2}(f,g)+h^{ 3}A_{3}(f,g)\right)\star h \\
  &=& (fg)\star h +hA_{1}(f,g)\star h+ h^{ 2}A_{2}(f,g)\star h+h^{ 3}A_{3}(f,g)\star h\\
   &=& fgh+hA_{1}(fg,h)+h^{ 2}A_{2}(fg,h)+h^{ 3}A_{3}(fg,h) \\
   && +h \left(A_{1}(f,g)h+hA_{1}\left(A_{1}(f,g),h\right)+h^{2}A_{2}\left(A_{1}(f,g),h\right)\right.\\
   &&\left.+h^{3}A_{3}\left(A_{1}(f,g),h\right)\right)+h^{2} \left(A_{2}(f,g)h+hA_{1}\left(A_{2}(f,g),h\right)\right.\\
   &&\left.+h^{2}A_{2}\left(A_{2}(f,g),h\right)+h^{3}A_{3}\left(A_{2}(f,g),h\right)\right)\\
&&+h^{3} \left(A_{3}(f,g)h+hA_{1}\left(A_{3}(f,g),h\right)+h^{2}A_{2}\left(A_{3}(f,g),h\right)\right.\\
&&\left.+h^{3}A_{3}\left(A_{3}(f,g),h\right)\right)\\
&=& fgh+hA_{1}(fg,h)+h^{ 2}A_{2}(fg,h)+h^{ 3}A_{3}(fg,h) \\
   && +h \left(A_{1}(f,g)h+hA_{1}\left(A_{1}(f,g),h\right)+h^{2}A_{2}\left(A_{1}(f,g),h\right)\right)\\
   &&+h^{2} \left(A_{2}(f,g)h+hA_{1}\left(A_{2}(f,g),h\right)\right)+h^{3} \left(A_{3}(f,g)h\right)\\
&=& fgh +h \left(A_{1}(fg,g)+A_{1}(f,g)h\right)+h^{2} \left(A_{1}\left(A_{1}(f,g),h\right)\right.\\
&&\left.+A_{2}(f,g)h\right)+h^{3}\left(A_{2}\left(A_{1}(f,g),h\right)+A_{1}\left(A_{2}(f,g),h\right)+A_{3}(f,g)h\right)\\
\\\\\\\\\\\\\\\\\\\\
\end{eqnarray*}
and
\begin{eqnarray*}
  f\star (g\star h) &=& f\star \left(gh+hA_{1}(g,h)+h^{ 2}A_{2}(g,gh)+h^{ 3}A_{3}(g,h)\right) \\
  &=& f\star(gh)+fh\star A_{1}(g,h)+ fh^{ 2}\star A_{2}(g,h)+fh^{ 3}\star A_{3}(g,h)\\
   &=& fgh+hA_{1}(f,gh)+h^{ 2}A_{2}(f,gh)+h^{ 3}A_{3}(f,gh) \\
   && +fhA_{1}(g,h)+hA_{1}\left(fh,A_{1}(g,h),h\right)+h^{2}A_{2}\left(fh,A_{1}(g,h),h\right)\\
   &&+h^{3}A_{3}\left(fh,A_{1}(g,h),h\right)+\\
   &&+fh^{2}A_{2}(g,h)+hA_{1}\left(fh^{2},A_{2}(g,h)\right)+h^{2}A_{2}\left(fh^{2},A_{2}(g,h)\right)\\
   &&+h^{3}A_{3}\left(fh^{2},A_{2}(g,h)\right)\\
&&+fh^{3}A_{3}(g,h)+hA_{1}\left(fh^{3},A_{3}(g,h)\right)+h^{2}A_{2}\left(fh^{3},A_{3}(g,h)\right)\\
   &&+h^{3}A_{3}\left(fh^{3},A_{3}(g,h)\right)\\
 &=& fgh+hA_{1}(f,gh)+h^{ 2}A_{2}(f,gh)+h^{ 3}A_{3}(f,gh) \\
   && +fhA_{1}(g,h)+hA_{1}\left(fh,A_{1}(g,h),h\right)+h^{2}A_{2}\left(fh,A_{1}(g,h),h\right)\\
   &&+fh^{2}A_{2}(g,h)+hA_{1}\left(fh^{2},A_{2}(g,h)\right)+fh^{3}A_{3}(g,h)\\
   &=& fgh +h \left(A_{1}(f,gh)+fA_{1}(g,h)h\right)+h^{2} \left(A_{2}(f,gh)\right.\\
   &&\left.+\left(f,A_{1}(g,h)\right)+fA_{2}(g,h)\right)\\
&&+h^{3}\left(A_{2}(f,gh)+A_{2}\left(f,A_{1}(f,gh)\right)+A_{1}\left(f,A_{2}(f,gh)\right)+fA_{2}(g,h)\right).
\end{eqnarray*}
It follows that
\begin{eqnarray*}
&& A_{1}(fg,h)+ A_{1}(f,g)h-A_{1}(f,gh)-fA_{1}(g,h) \\
 &&+h \left(A_{1}\left(A_{1}(f,g),h\right)-A_{1}\left(f,A_{1}(g,h)\right)+A_{2}(f,g)h-A_{2}(f,gh)-fA_{2}(g,h)\right) \\
&&h^{2}\left(A_{2}\left(A_{1}(f,g),h\right)+A_{1}\left(A_{2}(f,g),h\right)-A_{1}\left(f,A_{2}(g,h)\right)\right.\\
   &&\left.-A_{2}\left(f,A_{1}(g,h)\right)+A_{3}(f,g)h-A_{3}(f,gh)-fA_{3}(g,h)\right)=0.
\end{eqnarray*}
We obtain
\begin{eqnarray*}
 &&\delta_{KV}A_{1}(f,g,h)+h \left(\delta_{KV}A_{2}(f,g,h)+\frac{1}{2}[A_{1},A_{2}](f,g,h) \right)\\
 &&+
 h^{2}\left(\delta_{KV}A_{3}(f,g,h)+[A_{1},A_{2}](f,g,h) \right)=0.
 \end{eqnarray*}

\section{The generators}
The generators are given by\newline $J_{1}=\left(
         \begin{array}{ccccc}
           0 & 0 & 0 & 0 & 0 \\
           0 & 0 & -1 & 0 & 0 \\
           0 & 1 & 0 & 0 & 0 \\
           0 & 0 & 0 & 0 & 0 \\
           0 & 0 & 0 & 0 & 0 \\
         \end{array}
       \right)
,\quad J_{2}=\left(
         \begin{array}{ccccc}
           0 & 0 & 1 & 0 & 0 \\
           0 & 0 & 0 & 0 & 0 \\
           -1 & 0 & 0 & 0 & 0 \\
           0 & 0 & 0 & 0 & 0 \\
           0 & 0 & 0 & 0 & 0 \\
         \end{array}
       \right)
,\quad J_{3}=\left(
         \begin{array}{ccccc}
           0 & -1 & 0 & 0 & 0 \\
           1 & 0 & 0 & 0 & 0 \\
           0 & 0 & 0 & 0 & 0 \\
           0 & 0 & 0 & 0 & 0 \\
           0 & 0 & 0 & 0 & 0 \\
         \end{array}
       \right)$\newline $K_{1}=\left(
         \begin{array}{ccccc}
           0 & 0 & 0 & 1 & 0 \\
           0 & 0 & 0 & 0 & 0 \\
           0 & 0 & 0 & 0 & 0 \\
           0 & 0 & 0 & 0 & 0 \\
           0 & 0 & 0 & 0 & 0 \\
         \end{array}
       \right)
,\quad K_{2}=\left(
         \begin{array}{ccccc}
           0 & 0 & 0 & 0 & 0 \\
           0 & 0 & 0 & 1 & 0 \\
           0 & 0 & 0 & 0 & 0 \\
           0 & 0 & 0 & 0 & 0 \\
           0 & 0 & 0 & 0 & 0 \\
         \end{array}
       \right)
,\quad K_{3}=\left(
         \begin{array}{ccccc}
           0 & 0 & 0 & 0 & 0 \\
           0 & 0 & 0 & 0 & 0 \\
           0 & 0 & 0 & 1 & 0 \\
           0 & 0 & 0 & 0 & 0 \\
           0 & 0 & 0 & 0 & 0 \\
         \end{array}
       \right)$\newline $P_{1}=\left(
         \begin{array}{ccccc}
           0 & 0 & 0 & 0 & 1 \\
           0 & 0 & 0 & 0 & 0 \\
           0 & 0 & 0 & 0 & 0 \\
           0 & 0 & 0 & 0 & 0 \\
           0 & 0 & 0 & 0 & 0 \\
         \end{array}
       \right)
,\quad P_{2}=\left(
         \begin{array}{ccccc}
           0 & 0 & 0 & 0 & 0 \\
           0 & 0 & 0 & 0 & 1 \\
           0 & 0 & 0 & 0 & 0 \\
           0 & 0 & 0 & 0 & 0 \\
           0 & 0 & 0 & 0 & 0 \\
         \end{array}
       \right)
,\quad P_{3}=\left(
         \begin{array}{ccccc}
           0 & 0 & 0 & 0 & 0 \\
           0 & 0 & 0 & 0 & 0 \\
           0 & 0 & 0 & 0 & 1 \\
           0 & 0 & 0 & 0 & 0 \\
           0 & 0 & 0 & 0 & 0 \\
         \end{array}
       \right)$\newline $H=\left(
         \begin{array}{ccccc}
           0 & 0 & 0 & 0 & 0 \\
           0 & 0 & 0 & 0 & 0 \\
           0 & 0 & 0 & 0 & 0 \\
           0 & 0 & 0 & 0 & 1 \\
           0 & 0 & 0 & 0 & 0 \\
         \end{array}
       \right)$.\newline
       So, $\mathrm{SGal(3)}$ denote the special Galilei
group whose Lie algebra is generated by
$\left\{J_{i},K_{i},P_{i},H\right\}_{i=1,\dots,3}$ such that
$J_{i}:$ roatations, $K_{i}:$ galilean boosts, $P_{i}:$ spatial
translation and $H$: time translation. We write
$\mu(J_{i})=j_{i},\quad\mu(K_{i})=k_{i},\quad\mu(P_{i})=p_{i},\quad\mu(H)=E$,\quad$i=1,\dots
,3$.       The dual basis of generator is given by\newline
       $j_{1}=\left(
         \begin{array}{ccccc}
           0 & 0 & 0 & 0 & 0 \\
           0 & 0 & -1 & 0 & 0 \\
           0 & 1 & 0 & 0 & 0 \\
           0 & 0 & 0 & 0 & 0 \\
           0 & 0 & 0 & 0 & 0 \\
         \end{array}
       \right)
,\quad j_{2}=\left(
         \begin{array}{ccccc}
           0 & 0 & 1 & 0 & 0 \\
           0 & 0 & 0 & 0 & 0 \\
           -1 & 0 & 0 & 0 & 0 \\
           0 & 0 & 0 & 0 & 0 \\
           0 & 0 & 0 & 0 & 0 \\
         \end{array}
       \right)
,\quad j_{3}=\left(
         \begin{array}{ccccc}
           0 & -1 & 0 & 0 & 0 \\
           1 & 0 & 0 & 0 & 0 \\
           0 & 0 & 0 & 0 & 0 \\
           0 & 0 & 0 & 0 & 0 \\
           0 & 0 & 0 & 0 & 0 \\
         \end{array}
       \right)$\newline $k_{1}=\left(
         \begin{array}{ccccc}
           0 & 0 & 0 & 1 & 0 \\
           0 & 0 & 0 & 0 & 0 \\
           0 & 0 & 0 & 0 & 0 \\
           0 & 0 & 0 & 0 & 0 \\
           0 & 0 & 0 & 0 & 0 \\
         \end{array}
       \right)
,\quad k_{2}=\left(
         \begin{array}{ccccc}
           0 & 0 & 0 & 0 & 0 \\
           0 & 0 & 0 & 1 & 0 \\
           0 & 0 & 0 & 0 & 0 \\
           0 & 0 & 0 & 0 & 0 \\
           0 & 0 & 0 & 0 & 0 \\
         \end{array}
       \right)
,\quad k_{3}=\left(
         \begin{array}{ccccc}
           0 & 0 & 0 & 0 & 0 \\
           0 & 0 & 0 & 0 & 0 \\
           0 & 0 & 0 & 1 & 0 \\
           0 & 0 & 0 & 0 & 0 \\
           0 & 0 & 0 & 0 & 0 \\
         \end{array}
       \right)$\newline $p_{1}=\left(
         \begin{array}{ccccc}
           0 & 0 & 0 & 0 & 1 \\
           0 & 0 & 0 & 0 & 0 \\
           0 & 0 & 0 & 0 & 0 \\
           0 & 0 & 0 & 0 & 0 \\
           0 & 0 & 0 & 0 & 0 \\
         \end{array}
       \right)
,\quad p_{2}=\left(
         \begin{array}{ccccc}
           0 & 0 & 0 & 0 & 0 \\
           0 & 0 & 0 & 0 & 1 \\
           0 & 0 & 0 & 0 & 0 \\
           0 & 0 & 0 & 0 & 0 \\
           0 & 0 & 0 & 0 & 0 \\
         \end{array}
       \right)
,\quad p_{3}=\left(
         \begin{array}{ccccc}
           0 & 0 & 0 & 0 & 0 \\
           0 & 0 & 0 & 0 & 0 \\
           0 & 0 & 0 & 0 & 1 \\
           0 & 0 & 0 & 0 & 0 \\
           0 & 0 & 0 & 0 & 0 \\
         \end{array}
       \right)$\newline $E=\left(
         \begin{array}{ccccc}
           0 & 0 & 0 & 0 & 0 \\
           0 & 0 & 0 & 0 & 0 \\
           0 & 0 & 0 & 0 & 0 \\
           0 & 0 & 0 & 0 & 1 \\
           0 & 0 & 0 & 0 & 0 \\
         \end{array}
       \right)$.\\
 we have $\mu = (j, k, p, E) \in
\mathfrak{sgal}(3)^*$. Furthermore, we have
\begin{eqnarray*}
{}[J_{1},J_{1}] &=& 0,\quad [J_{1},J_{2}] = J_{3},\quad[J_{1},J_{3}] = 0\\
{}  [J_{2},J_{1}]&=&0,\quad [J_{2},J_{2}] = 0,\quad[J_{2},J_{3}] = J_{1}\\
{} [J_{3},J_{1}] &=& J_{2},\quad [J_{3},J_{2}] = 0,\quad[J_{3},J_{3}] = 0\\
{}[J_{1},K_{1}] &=& 0,\quad [J_{1},K_{2}] = K_{3},\quad[J_{1},K_{3}] = 0\\
{}  [J_{2},K_{1}] &=& -K_{3},\quad [J_{2},K_{2}] = 0,\quad[J_{2},K_{3}] = K_{1}\\
{} [J_{3},K_{1}] &=& K_{2},\quad [J_{3},K_{2}] = 0,\quad[J_{3},K_{3}] = 0\\
{}  [J_{1},P_{1}] &=& 0,\quad [J_{1},P_{2}] = P_{3},\quad[J_{1},P_{3}] = -P_{2}\\
{}  [J_{2},P_{1}] &=& -P_{3},\quad [J_{2},P_{2}] = 0,\quad[J_{2},P_{3}] = P_{1}\\
{} [J_{3},P_{1}] &=& P_{2},\quad [J_{3},P_{2}] = -P_{1},\quad[J_{3},P_{3}] = 0\\
{}[K_{1},K_{1}] &=& 0,\quad [K_{1},K_{2}] = 0,\quad[K_{1},K_{3}] = 0\\
{}  [K_{2},K_{1}] &=& 0,\quad [K_{2},K_{2}] = 0,\quad[K_{2},K_{3}] = 0\\
{} [K_{3},K_{1}] &=& 0,\quad [K_{3},K_{2}] = 0,\quad[K_{3},K_{3}] = 0\\
{} [H,P_{1}] &=& 0,\quad [H,P_{2}] = P_{3},\quad[H,P_{3}] = 0\\
 {} [H,K_{1}] &=& P_{1},\quad [H,K_{2}] =P_{2} ,\quad[H,K_{3}] = P_{3}\\
{}[P_{1},P_{1}] &=& 0,\quad [P_{1},P_{2}] = 0,\quad[P_{1},P_{3}] = 0\\
 {} [P_{2},P_{1}] &=& 0,\quad [P_{2},P_{2}] = 0,\quad[P_{2},P_{3}] = 0\\
{} [P_{3},P_{1}] &=& 0,\quad [P_{3},P_{2}] = 0,\quad[P_{3},P_{3}] = 0\\
{}  [K_{1},P_{1}] &=& 0,\quad [K_{1},P_{2}] = 0,\quad[K_{1},P_{3}] = 0\\
 {} [K_{2},P_{1}] &=& 0,\quad [K_{2},P_{2}] = 0,\quad[K_{2},P_{3}] = 0\\
 {}[K_{3},P_{1}] &=& 0,\quad [K_{3},P_{2}] = 0,\quad[K_{3},P_{3}] = 0.
\end{eqnarray*}
This means that the commutator between a velocity boost $K$ an a
momentum translation $P$ is zero, implying that boosting then
translating is equivalent to translating then boosting, which is
false. Translating then boosting is different from boosting then
translating. In \cite{barg, sax}, to correct this physical
contradiction, a new central generator is introduced $m$. The
generator of $\xi$ is given by\newline $J_{1}=\left(
         \begin{array}{cccccc}
           0 & 0 & 0 & 0 & 0 &0\\
           0 & 0 & -1 & 0 & 0&0 \\
           0 & 1 & 0 & 0 & 0 &0\\
           0 & 0 & 0 & 0 & 0&0 \\
           0 & 0 & 0 & 0 & 0&0 \\
            0 & 0 & 0 & 0 & 0&0 \\
         \end{array}
       \right)
,\quad J_{2}=\left(
         \begin{array}{cccccc}
           0 & 0 & 1 & 0 & 0&0 \\
           0 & 0 & 0 & 0 & 0 &0\\
           -1 & 0 & 0 & 0 & 0 &0\\
           0 & 0 & 0 & 0 & 0 &0\\
           0 & 0 & 0 & 0 & 0 &0\\
            0 & 0 & 0 & 0 & 0&0 \\
         \end{array}
       \right)
,\quad J_{3}=\left(
         \begin{array}{cccccc}
           0 & -1 & 0 & 0 & 0&0 \\
           1 & 0 & 0 & 0 & 0&0 \\
           0 & 0 & 0 & 0 & 0&0 \\
           0 & 0 & 0 & 0 & 0&0 \\
           0 & 0 & 0 & 0 & 0 &0\\
            0 & 0 & 0 & 0 & 0&0 \\
         \end{array}
       \right)$\newline $K_{1}=\left(
         \begin{array}{cccccc}
           0 & 0 & 0 & 1 & 0 &0\\
           0 & 0 & 0 & 0 & 0&0 \\
           0 & 0 & 0 & 0 & 0&0 \\
           0 & 0 & 0 & 0 & 0&0 \\
           m & 0 & 0 & 0 & 0&0 \\
            0 & 0 & 0 & 0 & 0&0 \\
         \end{array}
       \right)
,\quad K_{2}=\left(
         \begin{array}{cccccc}
           0 & 0 & 0 & 0 & 0&0 \\
           0 & 0 & 0 & 1 & 0&0 \\
           0 & 0 & 0 & 0 & 0 &0\\
           0 & 0 & 0 & 0 & 0&0 \\
           0 & m & 0 & 0 & 0&0 \\
            0 & 0 & 0 & 0 & 0&0 \\
         \end{array}
       \right)
,\quad K_{3}=\left(
         \begin{array}{cccccc}
           0 & 0 & 0 & 0 & 0 &0\\
           0 & 0 & 0 & 0 & 0&0 \\
           0 & 0 & 0 & 1 & 0&0 \\
           0 & 0 & 0 & 0 & 0&0 \\
           0 & 0 & m & 0 & 0&0 \\
            0 & 0 & 0 & 0 & 0&0 \\
         \end{array}
       \right)$\newline $P_{1}=\left(
         \begin{array}{cccccc}
           0 & 0 & 0 & 0 & 0&1 \\
           0 & 0 & 0 & 0 & 0&0 \\
           0 & 0 & 0 & 0 & 0 &0\\
           0 & 0 & 0 & 0 & 0 &0\\
           0 & 0 & 0 & 0 & 0&0 \\
            0 & 0 & 0 & 0 & 0&0 \\
         \end{array}
       \right)
,\quad P_{2}=\left(
         \begin{array}{cccccc}
           0 & 0 & 0 & 0 &0& 0 \\
           0 & 0 & 0 & 0 &0& 1 \\
           0 & 0 & 0 & 0 &0& 0 \\
           0 & 0 & 0 & 0 &0& 0 \\
           0 & 0 & 0 & 0 &0& 0 \\
            0 & 0 & 0 & 0 & 0&0 \\
         \end{array}
       \right)
,\quad P_{3}=\left(
         \begin{array}{cccccc}
           0 & 0 & 0 & 0 &0& 0 \\
           0 & 0 & 0 & 0 &0& 0 \\
           0 & 0 & 0 & 0 &0& 1 \\
           0 & 0 & 0 & 0 &0& 0 \\
           0 & 0 & 0 & 0 & 0&0 \\
            0 & 0 & 0 & 0 & 0&0 \\
         \end{array}
       \right)$\newline $H=\left(
         \begin{array}{cccccc}
           0 & 0 & 0 & 0 & 0&0 \\
           0 & 0 & 0 & 0 & 0&0 \\
           0 & 0 & 0 & 0 & 0&0 \\
           0 & 0 & 0 & 0 &0& 1 \\
           0 & 0 & 0 & 0 &0& 0 \\
            0 & 0 & 0 & 0 & 0&0 \\
         \end{array}
       \right)$.\newline
       So  we write
$\mu(J_{i})=j_{i},\quad\mu(K_{i})=k_{i},\quad\mu(P_{i})=p_{i},\quad\mu(H)=E$,\quad$i=1,\dots
,3$.       The dual basis of generator is given by\\
      $j_{1}=\left(
         \begin{array}{cccccc}
           0 & 0 & 0 & 0 & 0 &0\\
           0 & 0 & -1 & 0 & 0&0 \\
           0 & 1 & 0 & 0 & 0 &0\\
           0 & 0 & 0 & 0 & 0&0 \\
           0 & 0 & 0 & 0 & 0&0 \\
            0 & 0 & 0 & 0 & 0&0 \\
         \end{array}
       \right)
,\quad j_{2}=\left(
         \begin{array}{cccccc}
           0 & 0 & 1 & 0 & 0&0 \\
           0 & 0 & 0 & 0 & 0 &0\\
           -1 & 0 & 0 & 0 & 0 &0\\
           0 & 0 & 0 & 0 & 0 &0\\
           0 & 0 & 0 & 0 & 0 &0\\
            0 & 0 & 0 & 0 & 0&0 \\
         \end{array}
       \right)
,\quad j_{3}=\left(
         \begin{array}{cccccc}
           0 & -1 & 0 & 0 & 0&0 \\
           1 & 0 & 0 & 0 & 0&0 \\
           0 & 0 & 0 & 0 & 0&0 \\
           0 & 0 & 0 & 0 & 0&0 \\
           0 & 0 & 0 & 0 & 0 &0\\
            0 & 0 & 0 & 0 & 0&0 \\
         \end{array}
       \right)$\newline $k_{1}=\left(
         \begin{array}{cccccc}
           0 & 0 & 0 & 1 & 0 &0\\
           0 & 0 & 0 & 0 & 0&0 \\
           0 & 0 & 0 & 0 & 0&0 \\
           0 & 0 & 0 & 0 & 0&0 \\
           m & 0 & 0 & 0 & 0&0 \\
            0 & 0 & 0 & 0 & 0&0 \\
         \end{array}
       \right)
,\quad k_{2}=\left(
         \begin{array}{cccccc}
           0 & 0 & 0 & 0 & 0&0 \\
           0 & 0 & 0 & 1 & 0&0 \\
           0 & 0 & 0 & 0 & 0 &0\\
           0 & 0 & 0 & 0 & 0&0 \\
           0 & m & 0 & 0 & 0&0 \\
            0 & 0 & 0 & 0 & 0&0 \\
         \end{array}
       \right)
,\quad k_{3}=\left(
         \begin{array}{cccccc}
           0 & 0 & 0 & 0 & 0 &0\\
           0 & 0 & 0 & 0 & 0&0 \\
           0 & 0 & 0 & 1 & 0&0 \\
           0 & 0 & 0 & 0 & 0&0 \\
           0 & 0 & m & 0 & 0&0 \\
            0 & 0 & 0 & 0 & 0&0 \\
         \end{array}
       \right)$\newline $p_{1}=\left(
         \begin{array}{cccccc}
           0 & 0 & 0 & 0 & 0&1 \\
           0 & 0 & 0 & 0 & 0&0 \\
           0 & 0 & 0 & 0 & 0 &0\\
           0 & 0 & 0 & 0 & 0 &0\\
           0 & 0 & 0 & 0 & 0&0 \\
            0 & 0 & 0 & 0 & 0&0 \\
         \end{array}
       \right)
,\quad p_{2}=\left(
         \begin{array}{cccccc}
           0 & 0 & 0 & 0 &0& 0 \\
           0 & 0 & 0 & 0 &0& 1 \\
           0 & 0 & 0 & 0 &0& 0 \\
           0 & 0 & 0 & 0 &0& 0 \\
           0 & 0 & 0 & 0 &0& 0 \\
            0 & 0 & 0 & 0 & 0&0 \\
         \end{array}
       \right)
,\quad p_{3}=\left(
         \begin{array}{cccccc}
           0 & 0 & 0 & 0 &0& 0 \\
           0 & 0 & 0 & 0 &0& 0 \\
           0 & 0 & 0 & 0 &0& 1 \\
           0 & 0 & 0 & 0 &0& 0 \\
           0 & 0 & 0 & 0 & 0&0 \\
            0 & 0 & 0 & 0 & 0&0 \\
         \end{array}
       \right)$\newline $E=\left(
         \begin{array}{cccccc}
           0 & 0 & 0 & 0 & 0&0 \\
           0 & 0 & 0 & 0 & 0&0 \\
           0 & 0 & 0 & 0 & 0&0 \\
           0 & 0 & 0 & 0 &0& 1 \\
           0 & 0 & 0 & 0 &0& 0 \\
            0 & 0 & 0 & 0 & 0&0 \\
         \end{array}
       \right)$\\
 we have $\mu = (j, k, p, E) \in
\mathfrak{sgal}(3)^*$. Futhermore, we have
\begin{eqnarray*}
{}[J_{1},J_{1}] &=& 0,\quad [J_{1},J_{2}] = J_{3},\quad[J_{1},J_{3}] = 0\\
{}  [J_{2},J_{1}]&=&0,\quad [J_{2},J_{2}] = 0,\quad[J_{2},J_{3}] = J_{1}\\
{} [J_{3},J_{1}] &=& J_{2},\quad [J_{3},J_{2}] = 0,\quad[J_{3},J_{3}] = 0\\
{}[J_{1},K_{1}] &=& 0,\quad [J_{1},K_{2}] = K_{3},\quad[J_{1},K_{3}] = 0\\
{}  [J_{2},K_{1}] &=& -K_{3},\quad [J_{2},K_{2}] = 0,\quad[J_{2},K_{3}] = K_{1}\\
{} [J_{3},K_{1}] &=& K_{2},\quad [J_{3},K_{2}] = 0,\quad[J_{3},K_{3}] = 0\\
{}  [J_{1},P_{1}] &=& 0,\quad [J_{1},P_{2}] = P_{3},\quad[J_{1},P_{3}] = -P_{2}\\
{}  [J_{2},P_{1}] &=& -P_{3},\quad [J_{2},P_{2}] = 0,\quad[J_{2},P_{3}] = P_{1}\\
{} [J_{3},P_{1}] &=& P_{2},\quad [J_{3},P_{2}] = -P_{1},\quad[J_{3},P_{3}] = 0\\
{}[K_{1},K_{1}] &=& 0,\quad [K_{1},K_{2}] = 0,\quad[K_{1},K_{3}] =0\\
{}  [K_{2},K_{1}] &=& 0,\quad [K_{2},K_{2}] =0,\quad[K_{2},K_{3}]=0\\
{} [K_{3},K_{1}] &=& 0,\quad [K_{3},K_{2}] = 0,\quad[K_{3},K_{3}] = 0\\
{} [H,P_{1}] &=& 0,\quad [H,P_{2}] = 0,\quad[H,P_{3}] = 0\\
 {} [K_{1},H] &=& P_{1},\quad [K_{2},H] =P_{2} ,\quad[K_{3},H] = P_{3}\\
{}[P_{1},P_{1}] &=& 0,\quad [P_{1},P_{2}] = 0,\quad[P_{1},P_{3}] = 0\\
 {} [P_{2},P_{1}] &=& 0,\quad [P_{2},P_{2}] = 0,\quad[P_{2},P_{3}] = 0\\
{} [P_{3},P_{1}] &=& 0,\quad [P_{3},P_{2}] = 0,\quad[P_{3},P_{3}] = 0\\
{}  [K_{1},P_{1}] &=& M,\quad [K_{1},P_{2}] = 0,\quad[K_{1},P_{3}] = 0\\
 {} [K_{2},P_{1}] &=& 0,\quad [K_{2},P_{2}] = M,\quad[K_{2},P_{3}] = 0\\
 {}[K_{3},P_{1}] &=& 0,\quad [K_{3},P_{2}] = 0,\quad[K_{3},P_{3}] = M.
\end{eqnarray*}
Here, the parameter $m$ explicitly arises as the Bargmann central
extension embedded within the matrix realization of the special
Galilei algebra $\mathfrak{sgal}(3)$. We write $[K_{i},P_{j}]
=\delta_{ij}M$ with $M=\left(
                             \begin{array}{cccccc}
                                0 & 0 & 0 & 0 & 0&0 \\
                                0& 0 & 0 & 0 & 0&0 \\
                                0 & 0 & 0 & 0 & 0 &0\\
                                0 & 0 & 0 & 0 & 0 &0\\
                                0 & 0 & 0 & 0 & 0 &m\\
                                0 & 0 & 0 & 0 & 0 &0\\
                             \end{array}
                           \right).
$ So we obtain \begin{eqnarray*}
   [M,J_{1}] &=& [M,K_{1}]= [M,P_{1}]= [M,H]=0\\
 {}  [M,J_{2}] &=& [M,K_{2}]= [M,P_{2}]= 0\\
 {}  [M,J_{3}] &=& [M,K_{3}]= [M,P_{3}]= 0.
\end{eqnarray*}

\section{Lie Poisson brackets among these coordinate
functions} Using the definition \ref{poi}, we have
\begin{eqnarray*}
\{j_1, j_1\}(\mu) &=& \mu\left([J_1, J_1]\right)=\mu(0)=0, \quad
\{j_1, j_2\}(\mu) = \mu\left([J_1, J_2]\right)=\mu(J_3)=j_3\\
\{j_1, j_3\}(\mu) &=& \mu\left([J_1, J_3]\right)=\mu(0)=0,\quad
\{j_2,
j_1\}(\mu) = \mu\left([J_2, J_1]\right)=\mu(0)=0\\
\{j_2, j_2\}(\mu) &=& \mu\left([J_2, J_2]\right)=\mu(0)=0, \quad
\{j_2,
j_3\}(\mu) = \mu\left([J_2, J_3]\right)=\mu(J_1)=j_1\\
\{j_3, j_1\}(\mu) &=& \mu\left([J_3, J_1]\right)=\mu(J_2)=j_2, \quad
\{j_3, j_2\}(\mu) = \mu\left([J_3, J_2]\right)=\mu(0)=0 \\
 \{j_3,j_3\}(\mu)&=& \mu\left([J_3, J_3]\right)=\mu(0)=0,\quad
\{k_1, E\}(\mu) = \mu\left([K_1, H]\right)=\mu(P_1)=p_1\\
\{k_2, E\}(\mu) &=& \mu\left([K_2, H]\right)=\mu(P_2)=p_2, \quad
\{k_3, E\}(\mu) = \mu\left([K_3, H]\right)=\mu(P_3)=p_3\\
\{k_1, p_1\}(\mu) &=& \mu\left([K_1, P_1]\right)=\mu(M)=m, \quad
\{k_1, p_2\}(\mu) = \mu\left([K_1, P_2]\right)=\mu(0)=0\\
\{k_1, p_3\}(\mu) &=& \mu\left([K_1, P_3]\right)=\mu(0)=0,\quad
\{k_2, p_1\}(\mu) = \mu\left([K_2, P_1]\right)=\mu(0)=0\\
\{k_2, p_2\}(\mu) &=& \mu\left([K_2, P_2]\right)=\mu(M)=m, \quad
\{k_2, p_3\}(\mu) = \mu\left([K_2, P_3]\right)=\mu(0)=0\\
\{k_3, p_1\}(\mu) &=& \mu\left([K_3, P_1]\right)=\mu(0)=0, \quad
\{k_3, p_2\}(\mu) = \mu\left([K_3, P_2]\right)=\mu(0)=0\\
\{k_3, p_3\}(\mu) &=& \mu\left([K_3, P_3]\right)=\mu(M)=m,\quad
\{j_1, k_1\}(\mu) = \mu\left([J_1, K_1]\right)=\mu(0)=0\\
\{j_1, k_2\}(\mu) &=& \mu\left([J_1, K_2]\right)=\mu( K_3)= k_3,
\quad \{j_1, k_3\}(\mu) = \mu\left([J_1, K_3]\right)=\mu(0)=0
\end{eqnarray*}

\begin{eqnarray*}
\{j_2, k_1\}(\mu) &=& \mu\left([J_2, K_1]\right)=\mu(- K_3)=- k_3,
\quad \{j_2, k_2\}(\mu) = \mu\left([J_2, K_2]\right)=\mu(0)=0\\
\{j_2, k_3\}(\mu) &=& \mu\left([J_2, K_3]\right)=\mu(K_1)=k_1,\quad
\{j_3, k_1\}(\mu) = \mu\left([J_3, K_1]\right)=\mu(K_2)=k_2\\
\{j_3,k_2\}(\mu) &=& \mu\left([J_3, K_2]\right)=\mu(0)=0, \quad
\{j_3, k_3\}(\mu) = \mu\left([J_3, K_3]\right)=\mu(0)=0\\
\{j_1, p_1\}(\mu) &=& \mu\left([J_1, P_1]\right)=\mu(0)=0, \quad
\{j_1, p_2\}(\mu) = \mu\left([J_1, P_2]\right)=\mu(P_3)=p_3\\
\{j_1, p_3\}(\mu) &=& \mu\left([J_1,P_3]\right)=\mu(-P_2)=-p_2,\quad
\{j_2, p_1\}(\mu) =\mu\left([J_2, P_1]\right)=-p_3\\
\quad \{j_2, p_2\}(\mu) &=& \mu\left([J_2, P_2]\right)=\mu(0)=0,
\quad \{j_2, p_3\}(\mu) = \mu\left([J_2,
P_3]\right)=\mu(-P_1)=-p_1,\\
\{j_3, p_1\}(\mu) &=& \mu\left([J_3, P_1]\right)=\mu(P_2)=p_2, \quad
\{j_3, p_2\}(\mu) = \mu\left([J_3, P_2]\right)=\mu(-P_1)=-p_1\\
 \{j_3, p_3\}(\mu) &=& \mu\left([J_3, P_3]\right)=\mu(0)=0,\quad\{k_i, m\} = 0,\qquad i\in\{1,2,3\}\\
\{k_i, k_a\} &=& 0,\qquad i,a\in\{1,2,3\}.
\end{eqnarray*}

%
%
%
%

\bigskip
\intextsep 0pt
\begin{wrapfigure}{l}{30mm}
\includegraphics{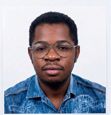}
\end{wrapfigure}

Dr. Prosper Rosaire Mama Assandje, the corresponding author,
received his PhD from the Department of Mathematics at the
University of Maroua in Cameroon. He is currently a lecturer at the
University of Yaounde I, Cameroon. His research interests include
Information Geometry, Symplectic Geometry, Differential Geometry,
Poisson Geometry, Mathematical Physics and KV-Cohomology.

His major contributions involve the theory and applications of
Completely Integrable Systems on Statistical Manifolds, Souriau's
Geometric Model of Statistical Mechanics, Souriau's Symplectic
Foliations, and Lie Group Cohomology.

Dr. Prosper Rosaire Mama Assandje can be contacted at\newline
 \emph{E-mail address}: {\tt mamarosaire@fasciences-uy1.cm}.\\
 \bigskip
University of Yaounde I, Faculty of Sciences,Department of
Mathematics,  Yaounde  P.O. Box 812, CAMEROON,
Tel:+237696477562/+237670217633

\bigskip
\intextsep 0pt
\begin{wrapfigure}{l}{30mm}
\includegraphics{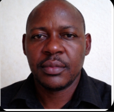}
\end{wrapfigure}

Dr. Romain Nimpa Pefoukeu received his PhD from the Department of
Mathematics at the University of Yaounde I in Cameroon. His research
interests include Geometry, Topology and Algebra.

His major contributions involve the theory and applications of the
prescribed Ricci curvature problem on five-dimensional nilpotent Lie
groups, the Z-decomposition of Euclidean Lie algebras, and locally
symmetric three-dimensional Riemannian Lie groups.

Dr. Romain Nimpa Pefoukeu can be contacted at\newline \emph{E-mail address}: {\tt romain.nimpa@facsciences-uy1.cm}.\\
University of Yaounde I, Faculty of Sciences,Department of
Mathematics,  Yaounde  P.O. Box 812, CAMEROON, Tel:+237699849360

\bigskip
\intextsep 0pt
\begin{wrapfigure}{l}{30mm}
\includegraphics{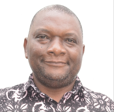}
\end{wrapfigure}

Dr. Michel Bertrand Djiadeu Ngaha received his PhD from the
Department of Mathematics at the University of Yaounde I in
Cameroon. His research interests include Riemannian Geometry,
Differential Geometry, Pure Mathematics, and Mathematical Analysis.

His major contributions involve the theory and applications of the
prescribed Ricci curvature problem on five-dimensional nilpotent Lie
groups, locally symmetric three-dimensional Riemannian Lie groups,
and Schouten-like metrics.

Dr. Michel Bertrand Djiadeu Ngaha can be contacted at\newline \emph{E-mail address}: {\tt michel.djiadeu@facsciences-uy1.cm}.\\
\bigskip
University of Yaounde I, Faculty of Sciences,Department of
Mathematics,  Yaounde  P.O. Box 812, CAMEROON, Tel:+237699943107

\bigskip
\intextsep 0pt
\begin{wrapfigure}{l}{30mm}
\includegraphics{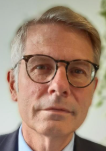}
\end{wrapfigure}

Dr. Frederic Barbaresco is a Senior THALES Expert in Artificial
Intelligence at the Technical Department of THALES Land Air Systems,
a SMART SENSORS Segment Leader for the THALES Corporate Technical
Department (Key Technology Domain "Processing, Control and
Cognition"), and a THALES Representative at the AI Expert Group of
ASD (AeroSpace and Defense Industries Association of Europe). He was
awarded the 2014 Aym\'ee Poirson Prize by the French Academy of
Science for the application of science to industry as well as the
Ampere Medal. He is an Emeritus Member of the SEE, President of the
SEE ISIC Club "Information and Communication Systems Engineering",
French MC Representative of the European COST CaLISTA and MSCA
CaLIGOLA, and General Chair of several elite and highly specialized
conferences, such as the SEE GSI "Geometric Science of Information".
His research interests mainly include Analytical Model Informed
Neural Networks: GINN (Geometry Informed Neural Network), PINN
(Physics Informed Neural Network), and TINN (Thermodynamics Informed
Neural Network).

Dr Frederic Barbaresco can be contacted at\newline \emph{E-mail address}: {\tt frederic.barbaresco@thalesgroup.com}.\\

\bigskip
Thales Land \& Air Systems(
Industry),https://www.thalesgroup.com/en, 2 Avenue Gay Lussac,
CS90502 78990 Elancourt, Region 8 (Africa, Europe, Middle East),
France. Thales Land \& Air Systems, Voie Pierre Gilles de Gennes,
F91470 Limours, France

\bigskip
\intextsep 0pt
\begin{wrapfigure}{l}{30mm}
\includegraphics{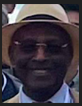}
\end{wrapfigure}
Professor Michel Nguiffo Boyom is a distinguished French-Cameroonian
mathematician and Professor Emeritus at the University of
Montpellier (Alexander Grothendieck Montpellier Institute - IMAG).
After earning his Doctorat d'Etat $\grave{e}$s Sciences from
University Paris-Sud Orsay in 1977, he established himself as an
internationally renowned expert in differential geometry, symplectic
geometry, Lie group theory, and information geometry, with notable
contributions to Hessian and Koszul structures.

Pr Michel Nguiffo Boyom can be contacted at\newline \emph{E-mail address}: {\tt nguiffo.boyom@gmail.com}.\\

\bigskip
Department of Mathematics, Institut Monpelli\'erain Alexander
Grothendieck (IMAG), imag-direction@umontpellier.fr,Montpellier,
34090 Montpellier,Case courrier 051 Place Eugene Bataillon, 499-554
Rue, du Truel 9, France

\label{last}
\end{document}